\documentclass[12pt]{amsart}
\usepackage{latexsym,amscd, amsmath, amssymb,epsfig,xypic,tikz, array} 
\usepackage{graphicx}
\usetikzlibrary{shapes,patterns,arrows,matrix,positioning,decorations.pathreplacing,patterns.meta}
\usepackage{float}
\usepackage{enumerate}
\usepackage{verbatim}

\theoremstyle{plain}
  \newtheorem{theorem}{Theorem}[section]
  \newtheorem{proposition}[theorem]{Proposition}
  \newtheorem{lemma}[theorem]{Lemma}
  \newtheorem{corollary}[theorem]{Corollary}
  \newtheorem{conjecture}[theorem]{Conjecture}
\theoremstyle{definition}
  \newtheorem{definition}[theorem]{Definition}
  \newtheorem{example}[theorem]{Example}

\theoremstyle{remark}

\numberwithin{equation}{section}

\def\FF{{\mathbb F}}
\def\NN{{\mathbb N}}

\def\ZZ{{\mathbb Z}}

\def\calP{{\mathcal P}}
\def\calS{{\mathcal S}}
\def\calX{{\mathcal X}}
\def\calY{{\mathcal Y}}

\def \sd{\mathop{\rm sd}} 

\newcommand{\SpCpx}[2]{\Delta(\mathcal{S}_{#2}(#1))}
\newcommand{\ApCpx}[2]{\Delta(\mathcal{A}_{#2}(#1))}
\newcommand{\BpCpx}[2]{\Delta(\mathcal{B}_{#2}(#1))}
\newcommand{\Sp}[2]{\mathcal{S}_{#2}(#1)}
\newcommand{\Ap}[2]{\mathcal{A}_{#2}(#1)}
\newcommand{\Bp}[2]{\mathcal{B}_{#2}(#1)}

\newcommand{\notdivide}{\kern-2.1pt\not\kern2.1pt\bigm|}
\newcommand{\bs}{\backslash}

\DeclareMathOperator \End{\mathop{\rm End}\nolimits}

\DeclareMathOperator \Hom{\mathop{\rm Hom}\nolimits}

\DeclareMathOperator \Res{\mathop{\rm Res}\nolimits}

\DeclareMathOperator \SB{\mathop{\rm SB}\nolimits}

\DeclareMathOperator{\Stab}{Stab}

\title{Synthetic Buildings For Finite Groups}
\author{Emily Gullerud}
\email{emily.gullerud@gmail.com}
\author{Peter Webb}
\email{webb@umn.edu}
\address{School of Mathematics\\
University of Minnesota\\
Minneapolis, MN 55455, USA}

\date{August 2026}

\begin{document}
\begin{abstract}

We introduce synthetic buildings for each finite group $G$ and prime $p$. These are $G$-simplicial complexes, among which are the $p$-subgroups complex of K.S. Brown, and also the buildings of finite groups of Lie type in characteristic $p$. Synthetic buildings are useful because of special properties, including providing a formula for the cohomology of $G$. Our goal is to describe the kinds of synthetic buildings that can arise, and ideally to classify them. The \textit{minimal} synthetic buildings play a special role and for most groups there are not very many of these. However, we show that it is possible to find sequences of finite groups with increasingly many minimal synthetic buildings as we move along the sequence, and also groups with minimal synthetic buildings of different dimensions. It is also possible to find groups with infinitely many non-minimal synthetic buildings of a given dimension. In other situations, we show that there are no minimal synthetic buildings of equivariant homotopy type distinct from those of Brown's complex and the complex consisting of a single point. We make a number of conjectures. 

\end{abstract}
\maketitle
\centerline{\it To the memory of Jon Alperin}


\section{Introduction}

Since the 1970s it has been known that the complex of non-identity $p$-subgroups of a finite group introduced by K.S. Brown plays a remarkable role in determining cohomological properties of the group. Other similar complexes have been introduced, for example by Quillen (using elementary abelian $p$-subgroups) and Bouc (using $p$-radical subgroups), as well as others, including complexes identified by Robinson and Alperin, and these all turn out to have the same equivariant homotopy type as Brown's complex. When this happens, it means that the cohomological and topological information obtained from all these complexes is substantially the same.

The difficulty in finding other complexes with similar properties, but of different equivariant homotopy types, has led us to ask how many there are. Whether there are only a few, or very many, the answer seems interesting in either case. To provide a framework in which to ask questions of this type, we introduce topological conditions that a simplicial complex with a group action may satisfy, and call such simplicial complexes \textit{synthetic buildings}. They include Brown's complex for each finite group and prime $p$.  Brown's complex is known to coincide up to equivariant homotopy equivalence with the building (in the sense of Tits) in the case of a group of Lie type in characteristic $p$, and this is the reason for the name we choose.

There are known results about synthetic buildings that were originally proved with complexes such as Brown's in mind, but  were stated in greater generality than this. Thus it is already known, for every finite group and prime $p$, that every synthetic building provides a  formula for the $p$-torsion part of the group cohomology. Over a complete $p$-local ring the augmented chain complex of a synthetic building is homotopic to a perfect complex. The Euler characteristic of a synthetic building is congruent to 1, modulo the order of a Sylow $p$-subgroup of the group. The orbit complex of a synthetic building is mod $p$ acyclic. We review these results in Section~\ref{definitions-section}  after making initial definitions. They are a justification for the study of synthetic buildings in the first place.

An important section in this paper is the last one, Section~\ref{conjectures-section}, in which we present a number of conjectures. Our conjectures have to do with the number of synthetic buildings that a group may have, and their structure, especially as it relates to Brown's complex. We also provide conjectures that strengthen Quillen's conjecture. It could be helpful to the reader to look at these conjectures early on, because they express our view of the shape of the theory.

We have been led to these conjectures by the theory presented in the earlier sections. It is immediate from our definitions that all finite groups of order divisible by $p$ have the possibility that a synthetic building may be just a single point. We term \textit{exotic} a synthetic building not of this form and not equivariantly homotopy equivalent to Brown's complex. In Section 3 we review some known examples of exotic synthetic buildings, pointing out properties that inform the conjectures. In Section 4 we develop some basics. For example it is not too hard to see that if the group is either a $p$-group or a group of order prime to $p$, the only possible synthetic building has the equivariant homotopy type of Brown's complex, and for a $p$-group this is a single point. The hardest part of this is done in Theorem~\ref{p-group-theorem}. Thus for $p$-groups there are no exotic synthetic buildings. It turns out that the same is true for many other groups as well.

For each finite group and prime $p$ the list of previously known examples of synthetic buildings has always been finite. In Section~\ref{infinite-family-subsection} we show (in more than one way) that a group may have infinitely many synthetic buildings, and even infinitely many of the same dimension. This possibility leads us to impose an extra condition on synthetic buildings, namely that they be minimal. A synthetic building is \textit{minimal} if it contains no invariant subcomplex that is also a synthetic building. In the infinite lists we construct, only finitely many of the synthetic buildings are minimal.

We show in Example~\ref{arbitrary-minimal-example} that it is possible to find pairs, consisting of a group and a prime $p$, for which there are arbitrarily many minimal synthetic buildings. By this we mean that given a natural number $n$, we can find a group $G_n$ that has at least $n$ minimal synthetic buildings at $p$. (The construction is given when $p=2$ but can be modified for arbitrary $p$.) However, in all the cases we consider the number of minimal synthetic buildings remains finite, and this is Conjecture~\ref{finiteness-conjecture}. We also show in Example~\ref{two-dimensions-example} that it is possible for a group to have minimal non-trivial synthetic buildings of different dimensions. In the other direction we develop a technique in Section~\ref{no-exotic-minimal-subsection} to show that in some cases there are no exotic minimal synthetic buildings. For instance, we prove in Corollary~\ref{Lie-rank-2-uniqueness} that this is so for groups of Lie type in characteristic $p$ in the case of Lie rank at most 2. Conjecture~\ref{BN-pair-conjecture} is that there are no exotic minimal synthetic buildings for groups of Lie type without restriction on the Lie rank. The truth of this conjecture would mean we have characterized buildings for groups of Lie type by our topological conditions and minimality.

Elsewhere in Sections 5 and 6 we characterize groups with a strongly $p$-embedded subgroup in terms of their synthetic buildings (Corollary~\ref{SPES-corollary}), and in Section~\ref{incidence-graphs-subsection} we characterize which of certain incidence graphs are synthetic buildings for the alternating and symmetric groups.

In Section~\ref{constructions-section} we develop several operations on synthetic buildings, and in particular we put the structure of a semigroup on the set of equivariant homotopy equivalence classes of them in two different ways (one of which is a monoid). In Section~\ref{Lefschetz-module-section} we provide a proof of a result stated in an earlier paper that gives a way to compute the Lefschetz modules of synthetic buildings. No proof was provided at the time. It has a consequence for the kind of projective modules that can appear as summands of the Lefschetz module.

We are grateful to Kevin Piterman for his careful reading of the manuscript and helpful comments that have simplified and clarified the exposition.

\section{Definitions}
\label{definitions-section}
Throughout, $G$ will be a finite group and $p$ a prime. By a \textit{$G$-simplicial complex} we mean a finite simplicial complex $\Delta$ with a simplicial action of $G$ such that each setwise simplex stabilizer fixes the simplex pointwise.

\begin{definition}
\label{synthetic-building-definition}
We say that a $G$-simplicial complex $\Delta$ is a \textit{synthetic building for $G$ at the prime $p$} if the following conditions hold:
    \begin{enumerate}
        \item The fixed point space $\Delta^P$ is $N_G(P)$-equivariantly contractible for all nonidentity $p$-subgroups $P\leq G$, and
        \item Every simplex stabilizer has order divisible by $p$.
    \end{enumerate}
We say that $\Delta$ is a \textit{weak synthetic building for $G$ at the prime $p$} if condition (1) is satisfied (and not necessarily condition (2)).
\end{definition}

We immediately observe that synthetic buildings are weak synthetic buildings, the property of being a weak synthetic building is preserved under $G$-homotopy equivalence, although the property of being a synthetic building need not be preserved under $G$-homotopy equivalence because the $G$-homotopy equivalence may introduce simplex stabilizers of order not divisible by $p$. The following is an equivalent way to state condition (1).

\begin{proposition}
\label{equiv-condition-proposition}
    Condition (1) is equivalent to the requirement 
\begin{enumerate}
    \item [(1$'$)] $\Delta^H$ is ordinarily contractible, for all subgroups $H\le G$ with $O_p(H)\ne 1$.
\end{enumerate}
\end{proposition}

\begin{proof}
    We use the fact that if $\Delta$ is a $K$-simplicial complex for some finite group $K$ then the map $f:\Delta\to \bullet$ is a $K$-equivariant homotopy equivalence if and only if for all subgroups $J\le K$ the restriction $f:\Delta^J\to \bullet$ is an ordinary homotopy equivalence. One direction of this is a theorem of Bredon \cite[Sect. II]{Bre1967}.

With this in mind, assume condition (1). If we let $H$ be a subgroup of $G$ with $O_p(H)\ne 1$ then $\Delta^H = (\Delta^{O_p(H)})^H $ and this is contractible because $\Delta^{O_p(H)}$ is equivariantly contractible for $N_G(O_p(H))$ and $H$ is a subgroup of $N_G(O_p(H))$.

Conversely, if we assume $\Delta^H\simeq \bullet$ whenever $O_p(H)\ne 1$ and $1\ne P$ is a $p$-subgroup we show $\Delta^P$ is $N_G(P)$-contractible, which is the same as the condition $\Delta^P\simeq_{N_G(P)/P}\bullet$. It suffices to show that $(\Delta^P)^K\simeq\bullet$ for all subgroups $K\le N_G(P)/P$. Writing $K=\bar K/P$, this is $\Delta^{\bar K}\simeq \bullet$, which holds because $P\le O_p(\bar K)$, so $O_p(\bar K)\ne 1$.
\end{proof}

We present the first examples of synthetic buildings.

\begin{example}
    If $p$ divides $|G|$, then the simplest case of a synthetic building for $G$ at $p$ is the one point space with the trivial action of $G$. Because this synthetic building exists for every finite group of order divisible by $p$, it does not provide any further information about the group. We sometimes refer to it as the \textit{trivial} synthetic building, and refer to synthetic buildings not of this equivariant homotopy type as \textit{non-trivial}.
\end{example}

\begin{example}
    If $p$ does not divide $|G|$ the only synthetic building for $G$ is the empty simplicial complex, by condition (2) of the definition. Every $G$-simplicial complex is a weak synthetic building in this situation. Furthermore, if a group $G$ has the empty set as a synthetic building, then $p\notdivide |G|$, by condition (1) of the definition. 
\end{example}

\begin{example}
We let $\Sp{G}{p}$ be the poset of all non-identity $p$-subgroups of $G$, ordered by inclusion, $\Ap{G}{p}$ the subposet of all non-identity elementary abelian $p$-subgroups and $\Bp{G}{p}$ the subposet of all non-identity $p$-radical subgroups, that is, subgroups $H\le G$ with $H=O_p(N_G(H))$. The corresponding order complexes $\SpCpx{G}{p}$, $\ApCpx{G}{p}$, and $\BpCpx{G}{p}$ were studied by Brown~\cite{Bro1975}, Quillen~\cite{Qui1978} and Bouc~\cite{Bou1984} respectively. For any finite group, these \textit{$p$-subgroup complexes} are all synthetic buildings at $p$. In fact, $\SpCpx{G}{p}$, $\ApCpx{G}{p}$, and $\BpCpx{G}{p}$ are all equivariantly homotopy equivalent \cite{The1991}, so it suffices first to show that just $\SpCpx{G}{p}$ is a weak synthetic building, and then also to show that these complexes satisfy condition (2) of Definition~\ref{synthetic-building-definition}. For the first part, condition ($1'$) of Proposition~\ref{equiv-condition-proposition} was verified in \cite[Lemma 2.1.2]{Web1991}. For condition (2), the simplices in any of these complexes are chains of non-identity $p$-subgroups and such a chain is always stabilized by the smallest group in the chain.

These complexes and others are reviewed by Smith in \cite{Smi2011}. There is a complex due to G.R. Robinson which is the subcomplex of $\SpCpx{G}{p}$ whose simplices are the chains $P_0<\cdots<P_n$ such that $P_i\triangleleft P_n$ for all $i$. In \cite{Alp1990} Alperin refers to three further simplicial complexes that have a separate history within group theory. These are the simplicial complexes where the simplices are (a) sets of Sylow $p$-subgroups with non-empty intersection, (b) sets of subgroups of order $p$ that generate a $p$-subgroup, and (c) sets of commuting elements of order $p$. All these complexes are equivariantly homotopy equivalent to $\SpCpx{G}{p}$ and satisfy condition (2) so are synthetic buildings.

\end{example}
\begin{example}
    The building in the sense of Tits of a finite group $G$ of Lie type in characteristic $p$ is a synthetic building for $G$ at $p$. In justification of this, it was claimed by Th\'evenaz-Webb~\cite{The1991} and confirmed by Piterman~\cite{Pit2024} that for such a group $G$, the complex $\BpCpx{G}{p}$ is the barycentric subdivision of the building, from which it follows that buildings are synthetic buildings. This means that the $p$-subgroup complexes can be regarded as a generalization of buildings to all finite groups, at every prime $p$, and hence synthetic buildings are also such a generalization.   It is this connection that motivates the terminology `synthetic building'. We conjecture in the last section that the only non-trivial synthetic buildings that are `minimal'  (in a sense to be defined) for groups of Lie type in characteristic $p$ are, in fact, buildings, so that buildings are conjecturally characterized in this way. In Corollary~\ref{Lie-rank-2-uniqueness} we prove that this is the case for groups of Lie rank at most 2. 
\end{example}

To justify the importance of (weak) synthetic buildings we recall some of their properties that are useful in understanding the cohomology and representation theory of finite groups.

\begin{theorem}
\label{structure-theorem-for-chain-complex}
    Let $\Delta$ be a weak synthetic building for a finite group $G$ at a prime $p$ and let $R$ be the $p$-adic integers $\ZZ_p$ or a field of characteristic $p$. Let $\tilde C.(\Delta)$ be the augmented chain complex of $\Delta$ over $R$. Then
    \begin{enumerate}
        \item $\tilde C.(\Delta)$ is chain homotopy equivalent to a finite length complex of finitely generated projective $RG$-modules. Explicitly, $\tilde C.(\Delta)\cong D.\oplus P.$ where $D.$ is a contractible complex, and $P.$ is a finite length complex of finitely generated projective $RG$-modules.
        \item The reduced Lefschetz module $\tilde L(\Delta) =\sum_{i\ge -1} \tilde C_i(\Delta)$ is a virtual projective module in the Green ring of finitely generated $RG$-modules.
    \end{enumerate}
\end{theorem}

Part (1) of this theorem was first proved in \cite[Thm. 2.7.1]{Web1991} and subsequently further proofs have been given in \cite[Cor. 7.10]{Bou1998} and \cite[Cor. 6.7]{Sym2005}. Part (2) is a consequence of part (1), but it was first proved earlier as Theorem A$'$ in \cite{Web1987-1} and discussed in \cite{Web1987-2}. When the synthetic building is $\SpCpx{G}{p}$ we call the complex of projectives $P.$ the \textit{Steinberg complex} for $G$ at $p$, and we call the reduced Lefschetz module $\tilde L(\Delta)$ the \textit{generalized Steinberg module} for $G$ at $p$ (although it is only a virtual module). When $G$ is a finite group of Lie type in characteristic $p$ it is the usual Steinberg module up to sign.


As a consequence of the fact that the reduced Lefschetz module is a virtual projective there arises a formula for group cohomology from any weak synthetic building, as was shown in \cite[Thm. A]{Web1987-1}. Such formulas take the following form that we can state with our new terminology. There is also a consequence for the Euler characteristic of $\Delta$ that follows from the fact that the ranks of projective $RG$ modules are always divisible by $|G|_p$, the largest power of $p$ that divides $|G|$. This fact about the Euler characteristic can also be proved in more elementary ways.
\begin{theorem}\label{thm:WebbA}
Let $\Delta$ be a weak synthetic building for $G$ at $p$, and let $M$ be a finitely generated $\ZZ G$-module.
Then
\begin{enumerate}
    \item 
$$\hat{H}^n(G,M)_p=\sum_{\sigma\in G\bs\Delta}(-1)^{\dim\sigma}\hat{H}^n(G_\sigma,M)_p$$
for any $\ZZ G$-module and any $n\in\ZZ$.
\item The Euler characteristic $\chi(\Delta)\equiv 1 \;(\textrm{mod}\; |G|_p)$.
\end{enumerate}
\end{theorem}

Here $\hat H$ denotes Tate cohomology, the subscript $p$ on groups denotes the $p$-torsion subgroup and $G_\sigma$ is the stabilizer of the simplex $\sigma$. The formulas hold in the Grothendieck group of finite abelian groups with relations given only by direct sum decompositions. Such a formula determines the group cohomology up to isomorphism. The formulas have been used extensively as a tool in computing the cohomology of specific finite groups, see \cite{Ade2004}. 

There is also a strengthened form of these formulas proved in \cite[Thm. A]{Web1991} where the cohomology groups appear as terms in a contractible complex, allowing computation of the ring structure in cohomology and without the restriction that the coefficient modules be finitely generated. Weak synthetic buildings satisfy the hypothesis of that theorem (on taking
$$
\calX= \hbox{all $p$-subgroups of $G$} \quad\hbox{and}\quad\calY = \hbox{the identity group}
$$
in the notation in force there), so that there are such contractible complexes of cohomology groups $\hat H^n(G_\sigma,M)_p$ for every weak synthetic building. Furthermore, Theorem B of \cite{Web1991} holds similarly when $\Delta$ is any weak synthetic building. There is a consequence of this theorem that is part of Corollary 2.6.1 of \cite{Web1991} that applies to all synthetic buildings, not just  $\Delta(\calS_p(G))$ as stated there. We say that a space is mod $p$ acyclic if its homology with $\FF_p$ coefficients is trivial.

\begin{corollary}\label{cor:modp}
Let $\Delta$ be a synthetic building for $G$ at the prime $p$. Then the orbit space $G\backslash \Delta$  is mod $p$ acyclic.
\end{corollary}

We present a refined version of this corollary in Section~\ref{properties-section}. This result in the case of $\Delta(\calS_p(G))$ prompted a conjecture that $G\backslash\Delta(\calS_p(G))$ should always be contractible, and this was proved by Symonds~\cite{Sym1998}, with subsequent proofs given, for example, in \cite{Bux1999, Gro2023, Ste2023}. In our final section we make an analogous conjecture for synthetic buildings in general.

Every weak synthetic building gives a cohomology formula, but some of these formulas are more useful than others. If a weak synthetic building fails to be a synthetic building because it has a simplex whose stabilizer has order not divisible by $p$, that simplex will contribute zero to the cohomology formula because the $p$-torsion in the cohomology of the stabilizer is zero. In formulas coming from synthetic buildings there may still be excess terms that cancel, giving rise to a smaller formula. This will happen if the synthetic building has a subcomplex that is also a synthetic building, because that subcomplex will also produce a cohomology formula using a subset of the terms of the larger formula. We thus focus on synthetic buildings without such subcomplexes.

\begin{definition}
We say that a synthetic building for $G$ at the prime $p$ is \textit{minimal} if it has no invariant subcomplex that is also a synthetic building for $G$ at $p$.
\end{definition}

Evidently every synthetic building contains a subcomplex that is a minimal synthetic building, and it may contain several, of different homotopy types (consider the cone of a non-trivial minimal synthetic building, which has both the non-trivial and the trivial synthetic buildings as subspaces). 
We note that the $p$-subgroups complex $\SpCpx{G}{p}$ usually is not minimal: it has as subcomplexes $\ApCpx{G}{p}$ and $\BpCpx{G}{p}$, and those also need not be minimal. The determination and properties of the equivariant homotopy types of minimal synthetic buildings is one of the main objects of this work.


While the focus on minimal synthetic buildings does not restrict the essentially different cohomology formulas that can be found, it does restrict the possibilities for the reduced Lefschetz module that arise. If $\Delta$ is a non-minimal synthetic building, its reduced Lefschetz module may be distinct from those of minimal synthetic buildings. There are many examples of this phenomenon. For instance, we will see later (Example~\ref{a5-families-example} and elsewhere) that at $p=2$ the alternating group $A_5$ has only two minimal synthetic buildings, but has infinitely many other synthetic buildings, some of which have distinct reduced Lefschetz modules.

In the same spirit of restricting the equivariant homotopy types of (weak) synthetic buildings that may arise, we make the next definition:
\begin{definition}
\label{slender-orbit-diagram-definitions}
We say that a $G$-simplicial complex $\Delta$ is \textit{slender} if  the orbit space $G\bs\Delta$ is contractible. We also define an \textit{orbit diagram} for $\Delta$ to be a picture of $G\backslash\Delta$ with the orbits of simplices labeled by representative orbit stabilizers.
\end{definition}

This additional property is chosen for a couple of reasons. First off, it is possessed by every $p$-subgroup complex $\SpCpx{G}{p}$ by a theorem of Symonds \cite{Sym1998}, and this is a key example.  Second, it allows us to exclude certain trivialities that would produce infinitely many uninteresting spaces. For example, let $G=C_2$ and $p=2$, and suppose $\Delta$ consists of two copies of any $G$-simplicial complex $\Omega$, identified at a base point. We let $G$ act on $\Delta$ by swapping the two copies of $\Omega$ and fixing the base point, so that $\Delta$ is a weak synthetic building. Then the orbit space $G\bs\Delta$ is homeomorphic to $\Omega$, which could have arbitrary structure as a simplicial complex. Thus weak synthetic buildings can be unnecessarily complicated without some further condition, such as slenderness.


\section{Further examples of synthetic buildings}

So far, all the examples presented of synthetic buildings have had the equivariant homotopy type of a single point or the $p$-subgroups complex $\SpCpx{G}{p}$ (which is empty in case $p$ does not divide $|G|$).

\begin{definition}
    We will say that a synthetic building for $G$ at $p$ is \textit{exotic} if it does not have the equivariant homotopy type of either a single point or the $p$-subgroups complex $\SpCpx{G}{p}$.
\end{definition}

We will show in later sections how to construct many examples of exotic synthetic buildings. We will see that it is possible to find groups with arbitrarily large (finite) numbers of minimal synthetic buildings (Example~\ref{arbitrary-minimal-example}), and also minimal non-trivial synthetic buildings of different dimensions (Example~\ref{two-dimensions-example}). We will also show how to construct infinitely many synthetic buildings for a single group, all of bounded dimension 1, but these will not be minimal (Theorem~\ref{thm:infiniteGgraphs}). We are particularly concerned with describing the minimal exotic synthetic buildings.

For our first two examples we review work of 
Ryba, Smith and Yoshiara~\cite{Ryb1990} and subsequently Smith and Yoshiara~\cite{Smi1997} who identified a number of exotic synthetic buildings for finite simple groups. 
    Their combined work surveys geometries of the kind that arose in the wake of the classification of finite simple groups, from the perspective of fixed-point contractibility. They typically only verified contractibility of fixed points under $p$-subgroups, but the extension to subgroups $H$ with $O_p(H)\ne 1$ is straightforward. We point out which of these examples are minimal, and in the minimal case we show for these examples that they are homomorphic images of Brown's complex.

\begin{example}\label{ex:c3}
    Among the examples surveyed by Ryba, Smith and Yoshiara is the $C_3$ geometry of Neumaier~\cite{Neu1984} for the alternating group $A_7$. This is the incidence geometry of three types of elements called `points', `lines', and `planes'. The set of points is the set of 7 symbols permuted by $A_7$. The lines are the $\binom{7}{3}=35$ subsets of points of size 3. The planes are one of the two orbits under $A_7$ of Fano projective planes consisting of 7 points and 7 lines. This geometry turns out to be a 2-dimensional slender synthetic building for $A_7$ at $p=2$ with orbit diagram:
    
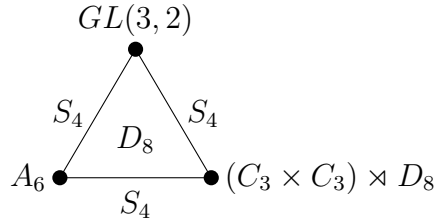
\begin{center}
        \begin{tikzpicture}
    \fill (0,0) circle (0.1)
        node[left=1pt] {$A_6$};
    \fill (2,0) circle (0.1)
        node[right=1pt] {$(C_3\times C_3)\rtimes D_8$};
   \fill (1,1.7) circle (0.1)
        node[above=1pt] {$GL(3,2)$};
    \draw (0,0)--(2,0)
        node[below=1pt,pos=0.5] {$S_4$};
   \draw (0,0)--(1,1.7)
        node[left=1pt,pos=0.5] {$S_4$};
   \draw (2,0)--(1,1.7)
        node[right=1pt,pos=0.5] {$S_4$};
\node at (1,0.5) {$D_8$};
\end{tikzpicture}
\end{center}
The three copies of $S_4$ shown are non-conjugate, and represent three of the four conjugacy classes of subgroups $S_4$ of $A_7$. Its reduced Lefschetz module is $\tilde L(\Delta) = +P_{20}$, the projective cover of the 20-dimensional simple $\FF_2A_7$-module.
        
They also identify another slender synthetic building for $A_7$ at $p=2$. This is a residue geometry of the Neumaier geometry and is explicitly constructed as the incidence graph of seven points permuted by $A_7$ with the subsets of those points of size 3, so it is a subgraph of the Neumaier geometry. It has orbit diagram
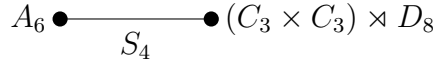
\begin{center}
\begin{tikzpicture}
    \fill (0,0) circle (0.1)
        node[left=1pt] {$A_6$};
    \fill (2,0) circle (0.1)
        node[right=1pt] {$(C_3\times C_3)\rtimes D_8$};
    \draw (0,0)--(2,0)
        node[below=1pt,pos=0.5] {$S_4$};
\end{tikzpicture}
\end{center}
This means the Neumaier geometry is not minimal as a synthetic building, although the residue geometry is: it is a graph with two orbits of vertices, neither of which forms a synthetic building by itself. Its reduced Lefschetz module is $\tilde L(\Delta) = -P_{14}$. Neither of these geometries is homotopy equivalent to $\SpCpx{A_7}{2}$ which was determined in \cite{Web1987-1} and is equivariantly homotopy equivalent to a graph with orbit diagram
    
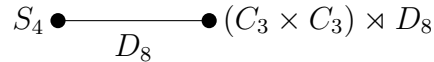
\begin{center}
\begin{tikzpicture}
    \fill (0,0) circle (0.1)
        node[left=1pt] {$S_4$};
    \fill (2,0) circle (0.1)
        node[right=1pt] {$(C_3\times C_3)\rtimes D_8$};
    \draw (0,0)--(2,0)
        node[below=1pt,pos=0.5] {$D_8$};
\end{tikzpicture}
\end{center}
Here the vertex stabilizers are the normalizers of the two conjugacy classes of subgroups of $A_7$ isomorphic to $C_2\times C_2$. The $S_4$ stabilizers for this graph and the previous one are not conjugate, and the Euler characteristics of the two graphs are distinct, so they are not homotopy equivalent. Its reduced Lefschetz module is $\tilde L(\Delta) = -(2P_{20}+P_{14})$.

It means that the incidence graph of points and 3-subsets is an exotic minimal synthetic building for $A_7$. It happens to be a homomorphic image of $\SpCpx{A_7}{2}$ by a homomorphism that is the identity on the vertex orbit with stabilizer $(C_3\times C_3)\rtimes D_8$, on the other vertex orbit is a mapping $A_7/S_4 \to A_7/A_6$ and on the edges is a mapping $A_7/D_8\to A_7/S_4$. 

To be more specific about this mapping, we can always realize a graph with orbit diagram 
\begin{center}
\begin{tikzpicture}
    \fill (0,0) circle (0.1)
        node[left=1pt] {$J$};
    \fill (2,0) circle (0.1)
        node[right=1pt] {$K$};
    \draw (0,0)--(2,0)
        node[below=1pt,pos=0.5] {$H$};
\end{tikzpicture}
\end{center}
as having vertices the cosets of $J$ and of $K$, with the edges being the cosets of $H$, an edge having vertices the cosets of $J$ and $K$ in which it is contained. Given a second such graph with stabilizers $J_1, K_1, H_1$ with $J\subseteq J_1$, $K\subseteq K_1$, and $H\subseteq H_1$, the mapping of graphs is determined by sending each coset of $J$ to the coset of $J_1$ that contains it, and similarly with the other subgroups.
\end{example}

\begin{example}
Something similar happens with the Mathieu group $M_{11}$.
Smith and Yoshiara~\cite{Smi1997} identify two geometries for $M_{11}$ that have orbit diagrams as follows.
\begin{center}
\begin{tikzpicture}
    \fill (0,0) circle (0.1)
        node[left=1pt] {$S_4$};
    \fill (2,0) circle (0.1)
        node[right=1pt] {$GL(2,3)$};
    \draw (0,0)--(2,0)
        node[below=1pt,pos=0.5] {$D_8$};
\end{tikzpicture}
\hskip.5in
\begin{tikzpicture}
    \fill (0,0) circle (0.1)
        node[left=1pt] {$M_{10}$};
    \fill (2,0) circle (0.1)
        node[right=1pt] {$GL(2,3)$};
    \draw (0,0)--(2,0)
        node[below=1pt,pos=0.5] {$SD_{16}$};
\end{tikzpicture}
\end{center}
Both of these are synthetic buildings, and the first of them is $\Delta(\calS_2(M_{11}))$, identified in \cite{Web1987-1}, with reduced Lefschetz module
$$\tilde L(\Delta(\calS_2(M_{11})))=-(3P_{44}+2P_{16}+2P_{\overline{16}}).
$$
The vertices in the first geometry can be taken to be the subgroups $C_2$ and $C_2\times C_2$, with an edge indicating containment of subgroups. The second geometry is an exotic minimal synthetic building, and its vertices can be taken to be the 11 points on which $M_{11}$ acts, and the subsets of those points of size 3. Computation shows that stabilizers can be chosen that are subgroups $S_4$ and $GL(2,3)$ of $M_{11}$, with $S_4$ a subgroup of $M_{10}$, so that $S_4\cap GL(2,3)$ is a subgroup $D_8$ and $M_{10}\cap GL(2,3)$ is a subgroup $SD_{16}$, and so these graphs may be realized as the coset graphs of these triples of subgroups. It follows that there is a surjective morphism of $G$-simplicial complexes from the first graph to the second, so that the second is a quotient of the first. The reduced Lefschetz module of the second complex is
$$\tilde L(\Delta)=-(2P_{44}+P_{16}+P_{\overline{16}}).$$
    \end{example}

\begin{example}
We mention some examples of spaces that are not weak synthetic buildings, but which illustrate why we choose to make the fixed point condition in the definition of a weak synthetic building so strong. At one point in this work we considered whether we should require simply that $\Delta^P$ be contractible for each nonidentity $p$-subgroup $P\le G$, in the definition of a weak synthetic building, and not equivariantly contractible for $N_G(P)$. That weaker condition would allow the complexes we now describe. We are grateful to Kevin Piterman for pointing out this construction.

 Let $H$ be any finite group acting on a compact contractible, but not H-contractible, complex $\Delta$. Such spaces are constructed, for example, in \cite{Oli1976}. Take a $p$-group $Q$ where $p$ does not divide the order of $H$, and let $Q$ act on $\Delta$ trivially. Let $G= Q\times H$ and let $G$ act on $\Delta$ by extending the actions of $H$ and $Q$. Then $\Delta$ satisfies that condition that $\Delta^P$ is contractible for each nonidentity $p$-subgroup $P\le G$, but it is not a synthetic building because $\Delta^Q$ is not $N_G(Q)$-contractible, and it is not $G$-homotopic to $\SpCpx{G}{p}$ (which is G-contractible). If complexes such as this were allowed to be synthetic buildings, they would be exotic, but they are not really of interest to us for applications to group cohomology and representation theory. We exclude them by our definition of synthetic building.
\end{example}


\section{Some properties of synthetic buildings}
\label{properties-section}

We start with some very basic observations about synthetic buildings. Observe first that if $\Delta_1$ is a weak synthetic building and $\Delta_1\simeq_G\Delta_2$ then $\Delta_2$ is also a weak synthetic building, and if $\Delta_1$ is slender then so is $\Delta_2$. On the other hand it is possible that $\Delta_1$ is a synthetic building, but $\Delta_2$ is not. As an example of this phenomenon consider the group of order 2 acting on a graph
\begin{tikzpicture}
    \fill (0,0) circle (0.1);
    \fill (1,0) circle (0.1);
    \fill (2,0) circle (0.1);
    \draw (0,0)--(2,0);
\end{tikzpicture}
by interchanging the two extreme vertices and fixing the middle vertex. This is a weak synthetic building with orbit diagram
    \begin{tikzpicture}
    \fill (0,0) circle (0.1)
        node[above=1pt] {$1$};
    \fill (2,0) circle (0.1)
        node[above=1pt] {$C_2$};
    \draw (0,0)--(2,0)
        node[above=1pt,pos=0.5] {$1$};
\end{tikzpicture}
but not a synthetic building. It is equivariantly homotopy equivalent to the synthetic building consisting of a single vertex and orbit diagram
    \begin{tikzpicture}
    \fill (0,0) circle (0.1)
        node[above=1pt] {$C_2$};
\end{tikzpicture}
We see from this example also that (weak) synthetic buildings can be described up to equivariant homotopy equivalence by different orbit diagrams, with different stabilizer subgroups.

A modification of this example where a cyclic group $C_4$ acts on the same graph with the the $C_2$ subgroup acting trivially, with orbit diagram     \begin{tikzpicture}
    \fill (0,0) circle (0.1)
        node[above=1pt] {$C_2$};
    \fill (2,0) circle (0.1)
        node[above=1pt] {$C_4$};
    \draw (0,0)--(2,0)
        node[above=1pt,pos=0.5] {$C_2$};
\end{tikzpicture},
is a synthetic building and it is equivariantly homotopy equivalent to the trivial synthetic building with a single vertex and orbit diagram
    \begin{tikzpicture}
    \fill (0,0) circle (0.1)
        node[above=1pt] {$C_4$};
\end{tikzpicture}.
 This shows that if two synthetic buildings are equivariantly homotopy equivalent, one might be minimal and the other not. In classifying minimal synthetic buildings up to equivariant homotopy equivalence, it will probably not be the case that every complex in the same equivariant homotopy equivalence class is also a minimal synthetic building.

We next explore the relationship between a synthetic building and its barycentric subdivision. Given a simplicial complex $\Delta$ we take its \textit{barycentric subdivision} $\sd \Delta$ to be the order complex of the poset of simplices of $\Delta$, ordered by inclusion. Thus the simplices of $\Delta$ are the vertices of $\sd\Delta$ and the (non-degenerate) $n$-simplices of $\sd\Delta$ are the sets of simplices $\{\sigma_0,\ldots,\sigma_n\}$ of $\Delta$ that are totally ordered under inclusions. Thus, perhaps after re-ordering, $\sigma_0\subset\ldots \subset\sigma_n$ is a proper chain of simplices of $\Delta$. Evidently if $\Delta$ is a $G$-simplicial complex then so is $\sd\Delta$. Part of the point of the next result is that we do not have to trouble ourselves by considering barycentric subdivisions in our basic definitions.

\begin{proposition}
\label{subdivision-proposition}
    Let $\Delta$ be a $G$-simplicial complex.
    \begin{enumerate}
        \item $\Delta$ is a weak synthetic building if and only if $\sd\Delta$ is a weak synthetic building.
        \item $\Delta$ is a synthetic building if and only if $\sd\Delta$ is a synthetic building.
 \item $\Delta$ is a minimal synthetic building if and only if $\sd\Delta$ is a minimal synthetic building.
    \end{enumerate}
\end{proposition}

\begin{proof}
    (1)  For each subgroup $H\le G$ we have $(\sd\Delta)^H=\sd(\Delta^H)$ so that one side is contractible if and only if so is the other.

(2) If the stabilizers of all simplices of $\sd\Delta$ have order divisible by $p$ then the same is true for $\Delta$, because the simplices of $\Delta$ are the vertices of $\sd\Delta$. On the other hand if the stabilizers of all simplices of $\sd\Delta$ have order divisible by $p$, because the stabilizer of a chain $\sigma_0\subset\ldots \subset\sigma_n$ equals the stabilizer of the largest simplex $\sigma_n$, the stabilizers of these chains also have order divisible by $p$.

       (3) We observe first the straightforward statement that if $\sd\Delta$ is minimal then so is $\Delta$. To see this, if $\Theta$ is a subcomplex of $\Delta$ that is a synthetic building then $\sd\Theta$ is a sub-synthetic building of $\sd\Delta$, which equals $\sd\Delta$ if $\sd\Delta$ is minimal. Thus if $\sd\Delta$ is minimal it follows that $\Theta=\Delta$ and $\Delta$ is minimal.

       To show that if $\Delta$ is minimal then so is $\sd\Delta$ is more elaborate. The condition that, for every simplex $\sigma$ of $\Delta$, the stabilizer $G_\sigma$ fixes $\sigma$ pointwise means that if we fix any $ \sigma$ and put its vertices in any total order, this can be extended to a partial order on the vertices of $\Delta$ so that it becomes a $G$-equivariant ordered simplicial complex, meaning that the vertices of each simplex are totally ordered, and the ordering is G-equivariant. This can be done by extending a relation $u<v$ on the vertices of $\sigma$ to $gu<gv$ on the vertices of $g\sigma$, and the stabilizer condition means this is well defined. We repeat this process with vertices of other simplices that have not already been put in order, to obtained an ordered simplicial complex, with the ordering equivariant for $G$.

Given any total ordering on the vertices of each simplex we obtain a simplicial map $g: \sd \Delta \to\Delta$ given by 
$$
g(\sigma_0\subset\cdots\subset\sigma_n) = \{\hbox{init}(\sigma_i)\bigm|0\le i\le n\}
$$
where $\hbox{init}(\sigma_i)$ is the least vertex of the simplex $\sigma_i$ of $\Delta$. Now $g$ is a homotopy equivalence, and if the ordering is G-equivariant so is the mapping $g$.

Suppose that $\Theta$ is a sub-synthetic building of $\sd \Delta$. We claim that $g(\Theta)$ is a sub-synthetic building of $\Delta$ for each of the mappings $g$. The fixed point condition $g(\Theta)^H\simeq \bullet$ for all subgroups $H$ with $O_p(H)\ne 1$  is satisfied because  $g$  is an equivariant homotopy equivalence. The stabilizers of simplices of $g(\Theta)$ have order divisible by $p$ because the same is true for $\Delta$. If we assume that  $\Delta$ is minimal then $g(\Theta) = \Delta$ for each of these mappings $g$. 

For each maximal simplex $\{x_0<x_1<\cdots <x_d\}$ of $\Delta$ there is a unique simplex of $\sd \Delta$ that maps to it under $g$, namely
$$\{(x_d) \subset (x_{d-1}< x_d)\subset\cdots\subset (x_1<\cdots <x_d)\subset (x_0<x_1<\cdots <x_d)\}.$$
Because $g$ surjects $\Theta$ onto $\Delta$, this maximal simplex of $\sd\Delta$ must belong to $\Theta$. Now, any maximal simplex of $\sd\Delta$ can be obtained in this way, by choosing a suitable ordering of the vertices of the simplex, and it follows that every maximal simplex of $\sd\Delta$ belongs to $\Theta$. From this we deduce $\sd\Delta=\Theta$ and $\sd\Delta$ is minimal.
\end{proof}

We now show that for $p$-groups the only synthetic buildings are trivial and we also characterize groups with $O_p(G)\ne 1$.

\begin{theorem}
\label{p-group-theorem}
Let $G$ be a finite group.
\begin{enumerate}
    \item If $G$ is a $p$-group then every synthetic building for $G$ at the prime $p$ is equivariantly contractible.
    \item $O_p(G)\ne 1$ if and only if every minimal synthetic building at the prime $p$ is equivariantly contractible.
\end{enumerate}
\end{theorem} 

\begin{proof} (1) This can be proved several ways. Perhaps the simplest approach is to observe that if $\Delta$ is a synthetic building for $G$ at $p$ then the mapping $\Delta\to\bullet$ is a $G$-homotopy equivalence, because for every subgroup  $J\le G$ the restriction $\Delta^J\to \bullet$ is an ordinary homotopy equivalence, and we can apply the theorem of Bredon \cite[Sect. II]{Bre1967}.

(2) Under the assumption that $O_p(G)\ne 1$, $\Delta^{O_p(G)}$ is contractible equivariantly for the action of $G = N_G(O_p(G))$, so that $\Delta$ has a point fixed by $G$. This fixed point is a subcomplex that is a synthetic building for $G$, so because $\Delta$ is minimal it equals this fixed point.

Conversely, if every minimal synthetic building building at $p$ is equivariantly contractible, consider a minimal $G$-subcomplex of $\SpCpx{G}{p}$. This has a fixed point under $G$, which is a normal $p$-subgroup. Thus $O_p(G)\ne 1$. 
\end{proof}

We continue with some observations about the orbit complex $G\backslash\Delta$ of a synthetic building. The following is a consequence of  \cite[Thm. A]{Web1991} using the argument of \cite[Cor. 2.6.1]{Web1991}.

\begin{corollary}
\label{orbit-space-mod-p-acyclic-corollary}
Let $\Delta$ be a synthetic building for $G$ at the prime $p$, and let $\calY$ be a non-empty set of $p$-subgroups of $G$, closed under taking subgroups and conjugation. Suppose $\calY$ does not contain a Sylow $p$-subgroup of $G$ and put 
$$
\Delta_\calY = \{\sigma \in \Delta\bigm| \hbox{Sylow $p$-subgroups of } G_\sigma \hbox{ do not lie in } \calY\}.
$$
Then $G\backslash \Delta_\calY$ is mod $p$ acyclic.
\end{corollary}

Taking $\calY=\{1\}$ is no extra condition when $\Delta$ is a synthetic building and shows that $G\backslash\Delta$ is always mod $p$ acyclic in this case.

\begin{proof}
Observe that $\Delta_\calY$ is a subcomplex of $\Delta$. In the notation of  \cite[Cor. 2.6.1]{Web1991} we let $\calX= \{ \hbox{all $p$-subgroups of }G\}$. To be able to apply \cite[Thm. A]{Web1991} we show that if $H\in\calX - \calY$ then $\Delta_\calY^H$ is contractible. This is so because $\Delta_\calY^H = \Delta^H$, which is contractible. To see this, we evidently have $\subseteq$, and if $\sigma\in\Delta^H$ then $H\subseteq G_\sigma$ and $H$ is a $p$-group, so $\sigma\in\Delta_\calY$, giving $\supseteq$. The argument now proceeds in the same way as  \cite[Cor. 2.6.1]{Web1991}.
\end{proof}

The following is a consequence of Corollary~\ref{orbit-space-mod-p-acyclic-corollary} but we give a simpler direct argument.

\begin{proposition}\label{prop:sb_ctd_orbit_space}
    If $\Delta$ is a synthetic building, then $G\bs\Delta$ is connected.
\end{proposition}
\begin{proof}
    Write $G\bs\Delta=G\bs\Theta_1\sqcup\cdots\sqcup G\bs\Theta_d$ as the connected components of $G\bs\Delta$, where $\Delta=\Theta_1\sqcup\cdots\sqcup\Theta_d$ (the $\Theta_i$ need not be connected). If $P$ is a non-identity $p$-subgroup of $G$, then $\Delta^P\simeq\bullet$ so that $\Delta^P\subseteq\Theta_i$ for some $i\leq d$. This means $\Theta_j^P=\emptyset$ if $i\ne j$, so every orbit of $P$ acting on $\Theta_j$ has size divisible by $p$, implying $p\mid \chi(\Theta_j)$. However, since $\Delta$ is a synthetic building, if $\sigma\in\Theta_j$ then $p\mid|\Stab(\sigma)|$. Thus there exists a non-identity $p$-subgroup $Q\leq G$ such that $\Theta_j^Q\simeq\bullet$ and it follows that $\chi(\Theta_j^Q)=1$. We have that $\chi(\Theta_j^Q)\equiv \chi(\Theta_j)\pmod{p}$, but this is impossible because $p\mid \chi(\Theta_j)$. Thus $i\ne j$ cannot happen, forcing $d=1$.
\end{proof}



\section{0-dimensional synthetic buildings and strongly $p$-embedded subgroups}

In this section we characterize the non-trivial synthetic buildings of dimension 0, showing that they arise precisely when the group has a strongly $p$-embedded subgroup. Such groups can be found with arbitrarily many minimal synthetic buildings of dimension 0, and also with minimal synthetic buildings of different dimensions. In a later section we will construct infinitely many (non-minimal) connected, exotic, synthetic buildings of the same dimension for a group with a strongly $p$-embedded subgroup.

\begin{definition}
    Let $H$ be a proper subgroup of $G$. Then $H$ is \textit{strongly $p$-embedded} in $G$ if and only if $p\mid |H|$ and $p\notdivide |H\cap {}^gH|$ whenever $g\not\in H$.
\end{definition}

It was shown by Quillen in \cite{Qui1978} that $\SpCpx{G}{p}$ is disconnected if and only if $G$ has a strongly $p$-embedded subgroup. We now extend this result.

\begin{theorem}
\label{SPES-theorem}
\begin{enumerate}
    \item If a group $G$ has a disconnected synthetic building then the setwise stabilizer of a component is a strongly $p$-embedded subgroup, and that component is a synthetic building for the stabilizer.
    \item If a group $G$ has a strongly $p$-embedded subgroup $K$ and $\Delta$ is a synthetic building for $K$, then $\Delta\uparrow_K^G$, defined to be the disjoint union of copies of $\Delta$ indexed by the cosets $G/K$, and with $G$-action given by the action on induced $G$-sets $(K/H)\uparrow_K^G\cong G/H$, is a synthetic building for $G$.
\end{enumerate}
\end{theorem}

\begin{proof}
    (1) Let $\Delta$ be a disconnected synthetic building for $G$ at $p$ and let $\Delta_0$ be a connected component, with set-wise stabilizer $H$. By Proposition~\ref{prop:sb_ctd_orbit_space} the orbit space $G\backslash \Delta$ is connected, so there is a single orbit of connected components under the action of $G$, and they are exactly the  complexes $x\Delta_0$, which are all distinct, as $x$ ranges through coset representatives of $H$ in $G$. If $x\not\in H$ then $H\cap{}^xH$ stabilizes both $\Delta_0$ and $x\Delta_0\ne\Delta_0$, so that if $g\in H\cap{}^xH$ has order $p$ then $\Delta^{\langle g\rangle}\supseteq \Delta_0^{\langle g\rangle} \cup x\Delta_0^{\langle g\rangle}$ which implies that $\Delta^{\langle g\rangle}$ is disconnected. This is not possible because such fixed points are required to be contractible, so $H\cap{}^xH$ has no element of order $p$, and $H$ is a strongly $p$-embedded subgroup.

    To show that $\Delta_0$ is a synthetic building for $H$: if $\sigma\in\Delta_0$ then $G_\sigma = H_\sigma$ and $p\bigm| |G_\sigma$ so $p\bigm| |H_\sigma$ also. Finally, if $K\le H$ with $O_p(K)\ne 1$ then $\Delta^K\subseteq x\Delta_0^K$ for some $x$, by connectivity, so $K\subseteq H\cap{}^xH$. But this group contains no elements of order $p$ unles $x\in H$, so $\Delta^K=\Delta_0^K$ is contractible.

    (2) Suppose that $\Delta$ is a synthetic building for the strongly $p$-embedded subgroup $K$. If $H$ is a $p$-subgroup of $G$ then $H\subseteq {}^xK$ where the coset $xK$ of $K$ in $G$ is uniquely determined, and $(\Delta\uparrow_K^G)^H\subseteq x\Delta$, which is a synthetic building for ${}^xK$. Thus $(\Delta\uparrow_K^G)^H= (x\Delta)^H$. Also $N_G(H) = N_{{}^xK}(H)$ and so $(x\Delta)^H$ is equivariantly contractible for $N_G(H)$, because $x\Delta$ is a synthetic building for ${}^xK$. Finally if $\sigma$ is a simplex of $x\Delta$ then the stabilizer $G_\sigma \supseteq ({}^xK)_\sigma$, which has order divisible by $p$.
\end{proof}

\begin{corollary}
\label{SPES-corollary}
    A group $G$ has a non-trivial synthetic building of dimension $0$ for the prime $p$ if and only if $G$ has a strongly $p$-embedded subgroup.
\end{corollary}

\begin{proof}
    This is immediate from Theorem~\ref{SPES-theorem}.
\end{proof}

\begin{example}
\label{arbitrary-minimal-example}
We construct groups with arbitrarily many minimal synthetic buildings. Our construction depends upon finding  groups with arbitrarily long chains of strongly $p$-embedded subgroups. 

For example, for each $i$ let the cyclic group $C_2$ act on an elementary abelian $3$-subgroup $L_i=C_3^i$ by inverting every element, and let $G_i=L_i\rtimes C_2$. Then $C_2$ is self-normalizing in $G_i$ and $\SpCpx{G_i}{2}$ is the disjoint union of $3^i$ points in a single orbit with stabilizer $C_2$, and it is a synthetic building for $G_i$. We can regard each $G_i$, $0\le i\le d-1$, as a homomorphic image of $G_{i+1}$ in such a way that we have a chain of surjective homomorphisms
    $$
    G_d\to G_{d-1}\to \cdots\to G_1\to G_0 = C_2.
    $$
Let $K_i$ be the kernel of the composite homomorphism $G_d\to G_i$. If we inflate $\SpCpx{G_i}{2}$ through the odd order group $K_i$, letting $K_i$ act trivially, we get a minimal synthetic building for $G_d$. It consists of a single orbit of $3^i$ points with stabilizer $K_i\rtimes C_2$ and is minimal because it is a single orbit. By this means we see that $G_d$ has $d+1$ minimal synthetic buildings that are not homotopy equivalent, and $d-1$ of which are exotic. They are all images of $\SpCpx{G_d}{2}$. Evidently $d$ can be made arbitrarily large.

The minimal synthetic buildings we have just constructed for this group are, in fact,  a complete list, as we now show.

    \begin{theorem}
        The minimal synthetic buildings at $p=2$ for the group $G_d=C_3^d\rtimes C_2$ are precisely the sets of points $G/(H\rtimes C_2)$ where $H\le C_3^d$. There are $d+1$ homotopically distinct minimal synthetic buildings.
    \end{theorem}

\begin{proof}
    We have just seen that these sets of points are minimal synthetic buildings. If $\Delta$ is a synthetic building for $G_d$ then $\Delta^{C_2}\simeq \bullet$ so there is a vertex $u\in\Delta$ with stabilizer containing $C_2$. The subgroups that contain $C_2$ are all of the form $H\rtimes C_2$ where $H\le C_3^d$, so $\Delta$ contains an orbit of vertices $G/(H\rtimes C_2)$. Because this orbit is a synthetic building and $\Delta$ is minimal, they must be equal.
\end{proof}
\end{example} 

\begin{example}
\label{two-dimensions-example}
We show that a group can have minimal non-trivial synthetic buildings of different dimensions.
    When $p$ is an odd prime the symmetric group $S_{2p}$ has a strongly $p$-embedded subgroup
    $$S_p\wr C_2=\langle (1,2,\ldots,p), (1,p+1)(2,p+2)\cdots(p,2p)\rangle$$ and the $p$-subgroups complex $\SpCpx{S_{2p}}{p}$ is the disjoint union of $\frac{(2p)!}{2(p!)^2}$ copies of $\SpCpx{S_p\times S_p}{p}$. The latter is the join of two copies of $\SpCpx{S_p}{p}$ which is a graph, and is not contractible if $p\ge 5$. From this we see that when $p\ge 5$ the $p$-subgroups complex $\SpCpx{S_{2p}}{p}$ is a minimal synthetic building of dimension 1. On the other hand, a discrete set of vertices in bijection with the cosets of $S_p\wr C_2$ also forms a 0-dimensional minimal synthetic building for $S_p\wr C_2$ that is exotic, by Theorem~\ref{SPES-theorem}, and it is a homomorphic image of $\SpCpx{S_{2p}}{p}$. This shows that, when $p\ge 5$, the group $S_{2p}$ has minimal non-trivial synthetic buildings of dimensions 0 and 1. We are grateful to Chris Parker for pointing out these groups.
\end{example}

We conclude this section by classifying the minimal synthetic buildings for many groups with a strongly $p$-embedded subgroup.

\begin{theorem}
\label{SPES-uniqueness-theorem}
    Suppose $G$ has a strongly $p$-embedded subgroup $H$ that is a maximal subgroup, and with $O_p(H)\ne 1$. Then the only minimal synthetic buildings are a one-point space $\bullet$ and $\Delta(\calS_p(G)) = G/H$.
\end{theorem}

\begin{proof}
Let $\Delta$ be a minimal synthetic building. Because $O_p(H)\ne 1$ we have that $\Delta^H$ is contractible and hence non-empty. Let $x\in\Delta^H$. Then the stabilizer of $x$ is a subgroup containing $H$, so it is $H$ or $G$ because $H$ is maximal. In these cases $\Delta$ contains a set of vertices that is a $G$-orbit $G/H$, or it has a fixed point. These are both synthetic buildings for $G$ (using Theorem~\ref{SPES-theorem} part (2)), so equal $\Delta$ by minimality, which completes the proof.
\end{proof}

Thus, for example, at $p=2$ neither $S_3$ nor $A_5$ has exotic minimal synthetic buildings.


\section{1-dimensional synthetic buildings}

In this section we do several things: we consider incidence graphs and introduce conditions under which they may be synthetic buildings, with specific examples for symmetric and alternating groups. For groups with minimal synthetic buildings of dimension 1 we develop a technique to show in some cases that no minimal synthetic buildings are exotic. Lastly, for groups with minimal non-trivial synthetic buildings of dimension 0 (hence, with a strongly $p$-embedded subgroup) we demonstrate the possibility of infinitely many synthetic buildings of dimension 1.

\subsection{Incidence graphs}
\label{incidence-graphs-subsection}

\begin{definition}
    Let $a,b,$ and $n$ be nonzero positive integers such that  $n>b\geq a\geq 1$. Let $\Delta=\Delta^n_{a,b}$ be the incidence graph whose vertices are the subsets of $[n]$ of size $a$ together with the subsets of $[n]$ of size $b$, and there is an edge joining an $a$-subset to a $b$-subset if and only if the $a$-subset is contained in the $b$-subset.
\end{definition}
Thus, the residue geometry of the Neumaier geometry in Example~\ref{ex:c3} is the $A_7$-graph $\Delta^7_{1,3}$. This is our motivating example for investigating incidence graphs further.

If we fix $G=A_n$, it turns out that $\Delta^n_{a,b}$
satisfies the condition to be slender for all choices of $a,b,n$ such that $n>b\geq a\geq 1$. The same result holds when enlarging $G$ to be all of $S_n$. However, we find that in very few cases $\Delta^n_{a,b}$ acted on by $G\in\{A_n,S_n\}$ satisfies the fixed point condition to be a synthetic building, leaving us with a handful of examples of incidence graphs that are slender synthetic buildings, which we classify.

We start by establishing results that are not dependent on the chosen group $G$, and hence would be useful in extending our results beyond the symmetric and alternating groups. The following lemma indicates that for a graph, is it sufficient to consider the fixed points under an element of order $p$ in order to determine the contractibility of the fixed points under any subgroup $H$ with $O_p(H)\ne 1$.

\begin{lemma}\label{lem:contractible_pelems}
Let $\Gamma$ be a $G$-graph. Then $\Gamma^H\simeq \bullet$ for all subgroups $H$ of $G$ with $O_p(H)\ne 1$ if and only if $\Gamma^{\langle g\rangle}\simeq \bullet$ for all elements $g$ of order $p$.
\end{lemma}
\begin{proof}
    The implication from left to right is immediate. For the implication from right to left we argue as follows. Every finite $p$-subgroup $P$ has a normal chain of subgroups $1=P_0\triangleleft P_1\triangleleft\cdots \triangleleft P_n=P$ so that each $P_i/P_{i-1}$ is cyclic of order $p$. Because $P_1$ is cyclic of order $p$, $\Gamma^P_1$ is a contractible graph so that it is a tree. Now, by induction, $\Gamma^{P_i}=(\Gamma^{P_{i-1}})^{P_i}$ is contractible since it is the fixed points of a finite group on a tree. Finally, if we take a subgroup $H$ with $O_p(H)\ne 1$ then $\Gamma^{O_p(H)}$ is contractible as just shown, and is a tree preserved by $H$. Thus $\Gamma^H$ is contractible because, again, it is the fixed points of a finite group acting on a tree.
\end{proof}

Now that we can focus solely on the action of a $p$-cycle on our complexes, the following theorem provides the architecture for finding synthetic buildings of the form $\Delta^n_{1,b}$ for any finite group $G$. We will subsequently apply this result to the symmetric and alternating groups.

\begin{theorem}\label{thm:fps_1b}
    Let $g$ be a permutation of order $p$ acting on $\Delta:=\Delta_{1,b}^n$. Let $f$ be the number of fixed points of $g$. Then $\Delta^{\langle g\rangle}\simeq \bullet$ if and only if at least one of the following holds:
    \begin{itemize}
        \item $f=1$ and $p\notdivide b$
        \item $f=b$ and either $b=p+1$ or $1\leq b\leq p-1$
    \end{itemize}
\end{theorem}
\begin{proof}
First suppose $f=1$. We show that $\Delta^{\langle g\rangle}\simeq \bullet$ if and only if $p\notdivide b$.
    Consider $b$-subsets $B$ fixed by $g$.
    If $p\mid b$ then such $B$ must be a union of $p$-cycles of $g$, together with fixed points of $g$, but since there is only one fixed point of $g$, $B$ must consist only of points moved by $g$. Thus $B$ is not incident with the single fixed point of $g$ and $\Delta^{\langle g\rangle}$ is disconnected if such $B$ exist. Observe that such $B$ do exist, because $b<n$. On the other hand, if $p\notdivide b$ then there is a single edge connecting the fixed point of $g$ to any such $B$ (there might not be any), which can be contracted down to this vertex. This completes the case $f=1$.  

    Now suppose $f=b$. Let $u_1,\ldots,u_b$ denote the fixed points of $g$. Then $g$ clearly fixes the $b$-subset $\{u_1,\ldots,u_b\}$ regardless of any restrictions on $b$ and $p$. If $1\leq b\leq p-1$, then this $b$-subset is the only fixed $b$-subset, meaning that $\Delta^{\langle g\rangle}$ is a cone and hence contractible. If $b=p+1$, then the other fixed $b$-subsets are of the form $\{u_i, x_1,\ldots, x_{b-1}\}$ for all $1\leq i \leq b$ and all choices of $p$-cycle $(x_1\,\,\cdots\,\,x_{b-1})$ of $g$. This means that if $g$ is comprised of a certain number $\lambda$ of cycles of size $p$, then each fixed point is incident with $\lambda$ subsets of size $b$, while each $b$-subset is incident with one and only one fixed point. Every one of these $b$-subsets contracts down to their respective fixed point, meaning that $\Delta^{\langle g\rangle}$ is again homotopy equivalent to a cone and hence contractible. On the other hand, if $b=p$, then there exists a fixed $b$-subset of the form $\{x_1,\ldots, x_b\}$ for each $p$-cycle $(x_1\,\,\cdots\,\,x_b)$ of $g$. Such a $b$-subset is not incident with any fixed points of $g$, resulting in $\Delta^{\langle g\rangle}$ being disconnected. If $b\geq p+2$, then there exist fixed $b$-subsets which contain two or more fixed points of $g$. Since such a $b$-subset is incident with multiple fixed points and the $b$-subset $\{u_1,\ldots,u_b\}$ is incident with every fixed point, $\Delta^{\langle g\rangle}$ contains at least one circuit. In summary, $\Delta^{\langle g\rangle}\simeq\bullet$ with $f=b$ if and only if $b=p+1$ or $1\leq b\leq p-1$.

    We will now show that if $f\ne 1$ and $f\ne b$, then $\Delta^{\langle g\rangle}\not\simeq\bullet$. Suppose $2\leq f<b$. If no $b$-subset is fixed by $g$, then $\Delta^{\langle g\rangle}$ is disconnected since $f\geq2$ and it is a disjoint union of this number of points. So suppose that at least one $b$-subset is fixed by $g$. Since $f<b$, every fixed $b$-subset must contain at least one non-fixed point. In particular, a fixed $b$-subset must consist of $k$ fixed points with $k\ge 1$, and $b-k$ non-fixed points which are a union of $p$-cycles of $g$. Suppose that $g$ consists of $\lambda$ cycles of size $p$ and $b-k=\gamma p$ , so $1\le\gamma\le \lambda$. If $\gamma<\lambda$, then there are $\binom{\lambda}{\gamma}$ fixed $b$-subsets which are incident with each set of $k$ fixed points, creating at least one circuit in $\Delta^{\langle g\rangle}$ if $k\geq 2$ and at least two disjoint components in $\Delta^{\langle g\rangle}$ if $k=1$ because this binomial coefficient is greater than 1. Now suppose $\gamma=\lambda$. In this situation, the fixed $b$-subset $B$ consists of all the $p$-cycles of $g$, together with some fixed points of $g$. We deduce $k<f$ because if $k=f$ then $b=n$, which has already been dealt with. If there is more than one fixed point in $B$, i.e. $k\ge 2$, then $\Delta^{\langle g\rangle}$ has a circuit. If there is only one fixed point in $B$, i.e. $k=1$, then $B$ is connected to that unique fixed point, and $\Delta^{\langle g\rangle}$ is disconnected because $f\ge 2$. In all cases, $\Delta^{\langle g\rangle}\not\simeq\bullet$.

    Finally, suppose $f>b$. If $b=1$ then $\Delta$ consists of a disjoint union of edges and $\Delta^{\langle g\rangle}\simeq \bullet$ if and only if $f=1$, which is accounted for in the first condition. Suppose that $f>b\geq2$.  Consider the subgraph of $\Delta^{\langle g\rangle}$ consisting of all $f$ fixed points of $g$ and all the $b$-subsets which contain only the fixed points of $g$. This subgraph has $b\cdot\binom{f}{b}$ edges and $f+\binom{f}{b}$ vertices. Since $b\geq2$, we have $\frac{b}{2}\cdot\binom{f}{b}\geq\binom{f}{b}$. Furthermore, $f>b$ implies that $\binom{f}{b}\geq f$. Thus, we have $\frac{b}{2}\cdot\binom{f}{b}+\frac{b}{2}\cdot\binom{f}{b}\geq f+\binom{f}{b}$ which simplifies to $b\cdot\binom{f}{b}\geq f+\binom{f}{b}$. Since the number of edges is equal to or exceeds the number of vertices, $\Delta^{\langle g\rangle}$ contains a circuit and hence is not contractible.
\end{proof} 

In these results we continue with our running assumption that $b<n$. We also suppose that $p\le n$ in what follows to ensure that there are some non-identity $p$-subgroups, thus avoiding the trivial case of a group of order prime to $p$ where every $G$-simplicial complex is a weak synthetic building.

With the goal of classifying all synthetic buildings of the form $\Delta^n_{1,b}$ for $G\in\{S_n,A_n\}$ at a prime $p$, we present the following two lemmas which provide the values of $a,b,p,n$ such that every simplex stabilizer in the $G$-graph $\Delta^n_{a,b}$ has order divisible by $p$. 

\begin{lemma}\label{lem:sn_stabsbyp}
    Let $\Delta=\Delta^n_{a,b}$ and let $S_n$ act on $\Delta$. Then every simplex is fixed by a non-identity $p$-subgroup of $S_n$ if and only if at least one of the following holds: $n-b\geq p$, $b-a\geq p$, or $a\geq p$.
\end{lemma}
\begin{proof}
Let $\ell_a$ be an $a$-subset of $[n]$ and $\ell_b$ be a $b$-subset of $[n]$.
Without loss of generality, let $\ell_a=\{n-a+1,\ldots,n\}$ and $\ell_b=\{n-b+1,\ldots,n\}$. Every element $g\in S_n$ which fixes $\ell_b$ is of the form $g=\sigma\cdot\tau$ where $\sigma\in S_{\{1,\ldots,n-b\}}$ and $\tau\in S_{\{n-b+1,\ldots,n\}}$ so that
$$\Stab(\ell_b)=\langle S_{\{1,\ldots,n-b\}}, S_{\{n-b+1,\ldots,n\}}\rangle.$$
The edge in $\Delta$ corresponding to the incidence $\ell_a\subseteq \ell_b$ has stabilizer
$$\Stab(\ell_a,\ell_b) = \langle S_{\{1,\ldots,n-b\}}, S_{\{n-a+1,\ldots,n\}}, S_{\{n-b+1,\ldots,n-a\}}\rangle.$$
To show that the order of $\Stab(\ell_a,\ell_b)$ is divisible by $p$ we show that at least one of these three symmetric groups has order divisible by $p$. This happens if and only $n-b\geq p$ or $a\geq p$ or $b-a\geq p$, which are the three conditions stated.
Since $\Stab(\ell_a,\ell_b)\subseteq\Stab(\ell_a)\cap \Stab(\ell_b)$, the desired result holds.
\end{proof}

\begin{lemma}\label{lem:an_stabsbyp}
    Let $\Delta=\Delta^n_{a,b}$ and let $A_n$ act on $\Delta$. If $p$ is an odd prime, then every simplex is fixed by a non-identity $p$-subgroup of $A_n$ if and only if at least one of the following holds: $n-b\geq p$, $b-a\geq p$, or $a\geq p$.
    
    If $p=2$, then every simplex is fixed by a non-identity $2$-subgroup of $A_n$ if and only if either at least two of $n-b\ge 2$, $b-a\ge 2$, $a\ge 2$ hold; or $(a,b)\in\{(1,1),(n-1,n-1)\}$ and $n\ge 5$; or $(a,b)\in\{(1,2), (1,n-1), (n-2,n-1)\}$ and $n\ge 6$.
\end{lemma}
\begin{proof}
Without loss of generality, let $\ell_a=\{n-a+1,\ldots,n\}$ be an $a$-subset of $[n]$ and $\ell_b=\{n-b+1,\ldots,n\}$ be a $b$-subset of $[n]$.
Every element $g\in A_n$ which fixes $\ell_b$ is of the form $g=\sigma\cdot\tau$ where $\sigma$ is a permutation of $\{1,\ldots,n-b\}$ and $\tau$ is a permutation of $\{n-b+1,\ldots,n\}$ such that $\sigma$ and $\tau$ have the same parity. If $\sigma$ and $\tau$ are both even permutations, then 
$\sigma\in A_{\{1,\ldots,n-a\}}$ and $\tau\in A_{\{n-a+1,\ldots,n\}}$. 
The case when $\sigma$ and $\tau$ are both odd permutations occurs only when $n\geq b+2$ and $b\geq2$. In this situation, we have $\sigma\in (1,2)\cdot A_{\{1,\ldots,n-b\}}$ and $\tau\in (n-b+1,n-b+2)\cdot A_{\{n-b+1,\ldots,n\}}$.  In summary, we have 
$$\Stab(\ell_b)=\begin{cases}
\langle A_{\{1,\ldots,n-b\}}, A_{\{n-b+1,\ldots,n\}}\rangle & \text{if }n=b+1 \text{ or }b=1\\
\langle A_{\{1,\ldots,n-b\}}, A_{\{n-b+1,\ldots,n\}}, C_2\rangle & \text{if }n\geq b+2\text{ and }b\geq 2\end{cases}$$
where $C_2$ is generated by the permutation $(1,2)(n-b+1,n-b+2)$.

We now find the stabilizer $\Stab(\ell_a,\ell_b)$ of the edge connecting $\ell_a$ and $\ell_b$. When $n-4\geq b-2\geq a\geq 2$, this is given by
$$\Stab(\ell_a,\ell_b)=\langle A_{\{1,\ldots,n-b\}}, A_{\{n-a+1,\ldots,n\}}, A_{\{n-b+1,\ldots,n-a\}}, C_2, C'_2, C''_2\rangle$$
where $C_2=\langle(1,2)(n-a+1,n-a+2)\rangle$, $C'_2=\langle(1,2)(n-b+1,n-b+2)\rangle$, and $C''_2=\langle(n-b+1,n-b+2)(n-a+1,n-a+2)\rangle$ so long as these subgroups are permitted by the values of $a, b$ and $n$.
Evidently $\Stab(\ell_a,\ell_b)$ has order divisible by $p$ if at least one of these six subgroups appearing in the stabilizer and has order divisible by $p$. We verify that if the stated conditions hold then one of these subgroups does appear. On the other hand, we may check that the finitely many cases excluded by the conditions do not give a stabilizer of order divisible by $p$.
Since $\Stab(\ell_a,\ell_b)\subseteq\Stab(\ell_a)\cap \Stab(\ell_b)$, the desired result holds.
\end{proof}

We now have all the tools necessary to classify all synthetic buildings of the form $\Delta^n_{1,b}$ for $S_n$ and $A_n$ at a prime $p$. Due to the acyclicity condition in Corollary~\ref{cor:modp}, all such synthetic buildings are also slender.

\begin{corollary}\label{cor:sbs_1b_sn}
    Let $p$ be a prime and $n\ge p$. Consider the action of $S_n$ on $\Delta=\Delta^n_{1,b}$. Then $\Delta$ is a synthetic building for $S_n$ at $p$ if and only if $n=p+b$ and either $b=p+1$ or $1\le b\le p-1$. In all cases, $\Delta$ is slender.
\end{corollary}
\begin{proof}
To prove the condition that $\Delta^H\simeq \bullet$ for all $p$-subgroups $H$, by Lemma~\ref{lem:contractible_pelems} it suffices to consider the condition that $\Delta^{\langle g\rangle}\simeq\bullet$ for all non-identity $p$-elements $g$. Observe that if $b<n=p$ then $\Delta^{\langle g\rangle}$ is empty, so is not contractible. Henceforth we may assume $p<n$.
    
Since we want $\Delta^{\langle g\rangle}\simeq \bullet$ for all non-identity $p$-elements of $S_n$, we need $\Delta^{\langle g\rangle}\simeq \bullet$ for all possible cycle types of $g$. For each $p\mid n!$, the possible cycle types for a $p$-element of $S_n$ consist of $p$-cycles where the number of $p$-cycles is 1 up to the largest multiple of $p$ less than $n$. 
    
Let $f$ be the number of points fixed by a given $p$-element $g$. Suppose $p<n<2p$. Then the only possible cycle type for $g$ is a single $p$-cycle. In order for $\Delta^{\langle g\rangle}\simeq \bullet$, we need $f=1$ or $f=b$ by Theorem~\ref{thm:fps_1b}. Note that $f=1$ if and only if $n=p+1$, and $f=b$ if and only if $n=p+b$. Thus $\Delta^{\langle g\rangle}\simeq \bullet$ with $p<n<2p$ if and only if $n=p+1$ and $p\notdivide b$ or $n=p+b$ and either $b=p+1$ or $1\leq b\leq p-1$ by Theorem~\ref{thm:fps_1b}. Now suppose $2p<n<3p$. Then there are two possible cycle types for $g$: a single $p$-cycle or two disjoint $p$-cycles. In order for $\Delta^{\langle g\rangle}\simeq \bullet$ for both choices of cycle type for $g$, one cycle types needs to produce $f=1$ and the other $f=b$. Necessarily the cycle type of a single $p$-cycle has $f=b$ and the cycle type of two $p$-cycles has $f=1$. This occurs if and only if $n=2p+1$ and $b=p+1$, which is the same as saying $n=p+b$ and $b=p+1$, which is one of the possibilities listed.

If $n>3p$, then $g$ has three possible cycle types: a single $p$-cycle, two disjoint $p$-cycles, or three disjoint $p$-cycles. Let $g_i$ denote a $p$-element with whose cycle type is $i$ disjoint $p$-cycles and let $f_i$ denote the number of points fixed by $g_i$. Necessarily none of the $f_i$ are equal to each other. Theorem~\ref{thm:fps_1b} states that only two choices for the value of $f_i$ could result in $\Delta^{\langle g\rangle}\simeq \bullet$, namely $1$ or $b$, so since the three $f_i$ are distinct and we need $\Delta^{\langle g_i\rangle}\simeq \bullet$ for all $i\in\{1,2,3\}$, we conclude that $\Delta^{\langle g\rangle}\not\simeq \bullet$ for at least one class of $p$-element $g$ when $n>3p$, and so $n>3p$ cannot occur.

If $n=kp$ for some $k>1$, then there is a $p$-element consisting of $k$ disjoint $p$-cycles which has no fixed points. Thus $\Delta^{\langle g\rangle}$ is empty and hence not contractible for every $p$-element $g$, so $n=kp$ cannot occur either.

In summary, $\Delta^{\langle g\rangle}\simeq \bullet$ for all non-identity $p$-elements of $S_n$ if and only if $n=p+1$ and $p\notdivide b$ or $n=p+b$ and either $b=p+1$ or $1\leq b\leq p-1$. 
In the former case when $n=p+1$, none of the inequalities in Lemma~\ref{lem:sn_stabsbyp} hold, so $\Delta$ has at least one simplex whose stabilizer has order prime to $p$. In the latter case when $n=p+b$, the inequality $n-b\geq p$ holds, so by Lemma~\ref{lem:sn_stabsbyp} every simplex in $\Delta$ is fixed by a $p$-subgroup. Hence, $\Delta$ is a synthetic building for $S_n$ at $p$ if and only if $n=p+b$ and either $b=p+1$ or $1\le b\le p-1$. 

If these conditions hold making $\Delta$ a synthetic building, then $G\bs\Delta$ is mod $p$ acyclic by Corollary~\ref{cor:modp}. This combined with the fact that $G\bs\Delta$ is a graph implies $G\bs\Delta$ is contractible. Thus, if $\Delta$ is a synthetic building it is necessarily slender.
\end{proof}

\begin{corollary}\label{cor:sbs_1b_an}
Let $n\ge p$ and consider the action of $A_n$ on $\Delta=\Delta^n_{1,b}$.
Then $\Delta$ is a synthetic building for $A_n$ at $p$
if and only if at least one of the following holds:
    \begin{enumerate}[(1)]
        \item $n=p+b$ and either $b=p+1$ or $1\le b\le p-1$, or 
        \item $n=7$, $p=2$, and $b=3$.
    \end{enumerate}
In both cases, $\Delta$ is slender.
\end{corollary}
\begin{proof}
     To prove that $\Delta^P\simeq \bullet$ for all nonidentity $p$-subgroups $P$, by Lemma~\ref{lem:contractible_pelems} it suffices to show that the claim holds for $P=\langle g\rangle$, for all non-identity $p$-elements $g$.

    If $p$ is odd, then the argument is the same as that of Corollary~\ref{cor:sbs_1b_sn} because $p$-cycles are even permutations (whose result is the first possibility listed). So suppose $p=2$. Since we want $\Delta^{\langle g\rangle}\simeq \bullet$ for all non-identity elements of order 2 of $A_n$, we need $\Delta^{\langle g\rangle}\simeq \bullet$ for all possible cycle types of $g$. These possible cycle types all consist of an even number of $p$-cycles on at most $n$ symbols. So that elements of order 2 exist, we take $n\geq 4$.

    Let $f$ be the number of points fixed by a given element $g$ of order 2. By Theorem~\ref{thm:fps_1b}, we have $\Delta^{\langle g\rangle}\simeq \bullet$ if and only if either $f=1$ and $b$ is odd or $f=b=3$. If $4\leq n<8$, then the only possible cycle type for $g$ is two disjoint $2$-cycles. Note that $f=1$ if and only if $n=5$ and $f=3$ if and only if $n=7$. In both of these cases, the necessary inequalities in Lemma~\ref{lem:an_stabsbyp} hold so that every simplex in $\Delta$ has stabilizer with order divisible by $p$.
    Thus $\Delta^5_{1,3}$ is a slender synthetic building for $A_5$ at $p=2$ (which is a special case of the first possibility listed) and $\Delta^7_{1,3}$ is a slender synthetic building for $A_7$ at $p=2$, while no further $\Delta^n_{1,b}$ are synthetic buildings for $A_5$, $A_6$ and $A_7$.

    Now suppose $n\geq8$. Then one of the possible cycle types of $g$ is two 2-cycles. Such a $g$ would have at least $f=4$ fixed points, meaning that $\Delta^n_{1,b}$ is not a synthetic building for $A_n$. Hence, the listed possibilities are the complete list for which $\Delta$ is a synthetic building.

    Just as in the previous corollary, if $\Delta$ is synthetic building, then  $G\bs\Delta$ is contractible. Thus, if $\Delta$ is a synthetic building it is slender.
\end{proof}

Corollaries~\ref{cor:sbs_1b_sn} and \ref{cor:sbs_1b_an} provides us with the constructions for a finite number of slender synthetic buildings for a given $S_n$ or $A_n$. More specifically, if we fix $n$ then, allowing the prime $p$ to vary, there exist $\pi(n)-\pi((\lfloor\frac{n-1}{2}\rfloor-1)$ unique slender synthetic buildings of the form $\Delta^n_{1.b}$ for both $S_n$ and $A_n$, where $\pi(n)$ is the prime-counting function, 
plus an additional synthetic building for $A_n$ when $n=7$. Unfortunately, it is not guaranteed that each of these synthetic buildings is minimal, as we see in the next example. 

\begin{example}
\label{Delta13-example}
Take $\Delta=\Delta_{1,3}^5$ acted on by $A_5$. By Corollary~\ref{cor:sbs_1b_an}, this is a slender synthetic building at $p=2$. This graph is a exotic synthetic building, though it is not minimal. Take $\Delta'$ to be the subgraph of $\Delta$ containing just the five 1-subsets. Each has stabilizer $A_4$, which is strongly 2-embedded in $A_5$, implying that $\Delta'$ is a synthetic building. A depiction of $\Delta$ is given below.
    
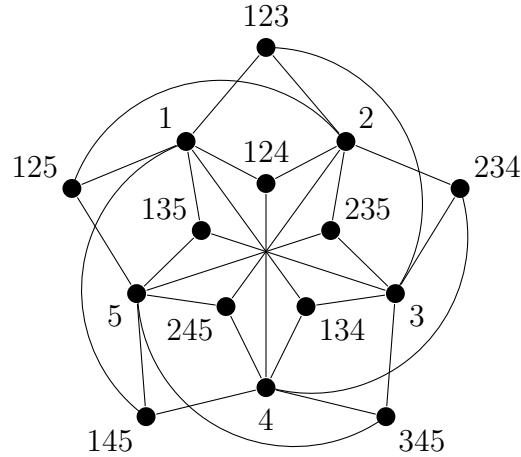
\begin{center}
        \begin{tikzpicture}[scale=0.9]
            \begin{scope}[vert/.style={circle, fill=black, thick, inner sep=2pt, minimum size=0.25cm}] 
                \node(124) at (0,1) [vert, label=above:124]{};
                \node(235) at (0.951,0.309) [vert, label={[label distance=-0.1cm]18:235}]{};
                \node(134) at (0.588,-0.809) [vert, label={[label distance=-0.1cm]306:134}]{};
                \node(245) at (-0.588,-0.809) [vert, label={[label distance=-0.1cm]234:245}]{};
                \node(135) at (-0.951,0.309) [vert, label={[label distance=-0.1cm]162:135}]{};
                \node(2) at (1.175,1.618) [vert, label={[label distance=-0.1cm]54:2}]{};
                \node(3) at (1.902,-0.618) [vert, label={[label distance=-0.1cm]342:3}]{};
                \node(4) at (0,-2) [vert, label=below:4]{};
                \node(5) at (-1.902,-0.618) [vert, label={[label distance=-0.1cm]198:5}]{};
                \node(1) at (-1.175,1.618) [vert, label={[label distance=-0.1cm]126:1}]{};
                \node(123) at (0,3) [vert, label=above:123]{};
                \node(234) at (2.853,0.927) [vert, label={[label distance=-0.1cm]18:234}]{};
                \node(345) at (1.763,-2.427) [vert, label={[label distance=-0.1cm]306:345}]{};
                \node(145) at (-1.763,-2.427) [vert, label={[label distance=-0.1cm]234:145}]{};
                \node(125) at (-2.853,0.927) [vert, label={[label distance=-0.1cm]162:125}]{};
            \end{scope}
            \draw (124) -- (1);
            \draw (124) -- (2);
            \draw (124) -- (4);
            \draw (235) -- (2);
            \draw (235) -- (3);
            \draw (235) -- (5);
            \draw (134) -- (1);
            \draw (134) -- (3);
            \draw (134) -- (4);
            \draw (245) -- (2);
            \draw (245) -- (4);
            \draw (245) -- (5);
            \draw (135) -- (1);
            \draw (135) -- (3);
            \draw (135) -- (5);
            \draw (123) -- (1);
            \draw (123) -- (2);
            \draw (234) -- (2);
            \draw (234) -- (3);
            \draw (345) -- (3);
            \draw (345) -- (4);
            \draw (145) -- (4);
            \draw (145) -- (5);
            \draw (125) -- (5);
            \draw (125) -- (1);
            \draw (123) arc (90:-32:2.3);
            \draw (234) arc (18:-104:2.3);
            \draw (345) arc (-54:-176:2.3);
            \draw (145) arc (-126:-248:2.3);
            \draw (125) arc (-198:-320:2.3);
        \end{tikzpicture}
    \end{center}
\end{example}

\subsection{A criterion for the absence of exotic minimal synthetic buildings}
\label{no-exotic-minimal-subsection}

An orbit diagram for a $G$-simplicial complex was already defined in Definition~\ref{slender-orbit-diagram-definitions}. In what follows, in a special case when the simplicial complex has dimension 1, it will be useful to make a more precise definition. 

\begin{definition}
Let $\Gamma$ be a $G$-graph for which $G\bs\Gamma$ is a tree.
We define an \textit{orbit diagram} for $\Gamma$ to be a subtree of $\Gamma$ whose vertices and edges form a set of orbit representatives for the orbits of $G$ on $\Gamma$, and where the stabilizer in $G$ of each vertex and edge is shown by the corresponding vertex and edge.
\end{definition}

It turns out that an orbit diagram that is a tree in this sense uniquely determines the $G$-graph $\Gamma$ up to isomorphism. This could be verified by a direct argument, but if we make reference to the machinery of scwols (small categories without loops), of Bridson and Haefliger in~\cite{Bri1999}, the result follows immediately.
\begin{theorem}\label{thm:uniqueGgraphs}
    Consider an orbit diagram for a $G$-graph which is a tree. Then there exists a unique $G$-graph with this orbit diagram, up to $G$-isomorphism.
\end{theorem}
\begin{proof}
    The result follows directly from Theorem 2.13(1) in \cite{Bri1999}.
\end{proof}


We now present a theorem whose goal is to prove that in many cases there are no exotic minimal synthetic buildings.

\begin{theorem}
\label{uniqueness-theorem}
    Let $\Delta$ be a non-trivial minimal synthetic building for $G$ at $p$. Let $J,K$ be non-conjugate maximal subgroups of $G$ and $H=J\cap K$. Suppose that
    \begin{enumerate}
        \item $\Delta^J$ and $\Delta^K$ are non-empty,
        \item $\Delta^H$ is path-connected.
        \item If $H\subset A$ is a subgroup that properly contains $H$ then $A$ is one of $K,J$ or $G$,
        \item The $G$-graph $\Gamma$ with orbit diagram 
        \scalebox{0.75}
            {\begin{tikzpicture}
    \fill (0,0) circle (0.1)
        node[above=1pt] {$J$};
    \fill (2,0) circle (0.1)
        node[above=1pt] {$K$};
    \draw (0,0)--(2,0)
        node[above=1pt,pos=0.5] {$H$};
\end{tikzpicture}}
is a synthetic building for $G$ at $p$.
    \end{enumerate}
Then $\Delta\simeq_G \Gamma$.
\end{theorem}

\begin{proof}
The conditions imply that $\emptyset\ne \Delta^J\subseteq\Delta^H$, $\emptyset\ne \Delta^K\subseteq\Delta^H$, and that any simplex $\sigma\in\Delta^H$ has its stabilizer equal to precisely one of $H,J$ or $K$. Let $x$ be a vertex in $\Delta^J$ and $y$ a vertex in $\Delta^K$. Because $\Delta^H$ is path-connected there is a path of edges in $\Delta^H$ of the form
\begin{center}    
\begin{tikzpicture}
    \fill (0,0) circle (0.1)
        node[left=1pt] {$x$};
    \fill (1,0) circle (0.1)
        node[above=1pt] {$a_1$};
    \draw (0,0)--(1,0);
    \draw (1,0)--(2,0);
    \node at (2.4,0) {$\ldots$};
    \draw (2.8,0)--(3.8,0);
    \fill (3.8,0) circle (0.1)
        node[above=1pt] {$a_n$};
    \draw (3.8,0)--(4.8,0);
    \fill (4.8,0) circle (0.1)
        node[right=1pt] {$y$};
\end{tikzpicture}
\end{center}
     with stabilizer groups $G_{a_i}$ containing $H$, so these stabilizers must be $J,K$ or $H$. Choose $x$ and $y$ so as to have a shortest such path. Now $G_{a_i}=H$ for all $i$, and all edge stabilizers are $H$. If some $a_i$ and $a_j$ with $i<j$ lie in the same $G$-orbit, so $a_j=ga_i$ for some $g\in G$, then we can shorten the path to
\begin{center}    
\begin{tikzpicture}
    \fill (0,0) circle (0.1)
        node[left=1pt] {$x$};
    \draw (0,0)--(1,0);
    \node at (1.4,0) {$\ldots$};
    \draw (1.8,0)--(2.8,0);
    \fill (2.8,0) circle (0.1)
        node[above=1pt] {\footnotesize$a_i=g^{-1}a_j$};
    \draw (2.8,0)--(4.8,0);
    \fill (4.8,0) circle (0.1)
        node[above=1pt] {\footnotesize$g^{-1}a_{j+1}$};
    \draw (4.8,0)--(5.8,0);
    \node at (6.2,0) {$\ldots$};
    \draw (6.6,0)--(7.6,0);
    \fill (7.6,0) circle (0.1)
        node[right=1pt] {$g^{-1}y$};
\end{tikzpicture}
\end{center}
    In doing this, note that the stabilizers of $a_i$ and $a_j=ga_i$ are both $H$, and the stabilizer of $ga_i$ can also be written ${}^gH$, so $H={}^gH$ and $g\in N_G(H)$. Thus all the $g^{-1}a_k$ have stabilizer $H$.
    Because the path was chosen to be shortest, we deduce that all the $a_i$ lie in distinct $G$-orbits. By taking the union of the $G$-orbits of this path we construct a $G$-subcomplex of $\Delta$ with orbit diagram
\begin{center}
\begin{tikzpicture}
    \fill (0,0) circle (0.1)
        node[left=1pt] {$J$};
    \fill (2,0) circle (0.1)
        node[above=1pt] {$H$};
    \draw (0,0)--(2,0)
        node[above=1pt,pos=0.5] {$H$};
    \draw (2,0)--(3,0);
    \node at (3.4,0) {$\ldots$};
    \draw (3.8,0)--(4.8,0);
    \fill (4.8,0) circle (0.1)
        node[above=1pt] {$H$};
    \fill (6.8,0) circle (0.1)
        node[right=1pt] {$K$};
    \draw (4.8,0)--(6.8,0)
        node[above=1pt,pos=0.5] {$H$};
\end{tikzpicture}
\end{center}
which is equivariantly homotopy equivalent to a complex with the same orbit diagram as $\Gamma$. By Theorem~\ref{thm:uniqueGgraphs}, that complex is $G$-isomorphic to $\Gamma$, so it is a synthetic building. Minimality of $\Delta$ implies $\Delta$ is this subcomplex.
\end{proof}

\begin{corollary}
\label{Lie-rank-2-uniqueness}
    Let $G$ be a finite group of Lie type in characteristic $p$ of Lie rank 2. Then $G$ has a unique non-trivial minimal synthetic building at $p$ up to $G$-homotopy equivalence, namely its building in the sense of Tits.
\end{corollary}

\begin{proof}
    Take $H$ to be a Borel subgroup and let $J$, $K$ be the two maximal parabolic subgroups that contain $H$. If $\Delta$ is a non-trivial minimal synthetic building then the conditions of Theorem~\ref{uniqueness-theorem} are satisfied for the following reasons. Conditions (1) and (2) hold because $H,J,K$ all have $O_p\ne 1$ so the fixed points are contractible. Condition (3) holds by standard properties of parabolic subgroups and (4) holds because $\Gamma$ is the building for $G$ in the sense of Tits. The corollary is now a consequence of Theorem~\ref{uniqueness-theorem}.
\end{proof}

This corollary applies to a number of small groups, such as $GL(3,2)$, $A_6$ and $S_6$ at $p=2$, sometimes through exceptional isomorphisms with groups of Lie type. We can extend the argument of Theorem~\ref{uniqueness-theorem} to more general situations, exemplified in the following theorem.

\begin{theorem}
    At $p=2$ the symmetric group $S_5$ has a unique non-trivial minimal synthetic building.
\end{theorem}

The non-trivial minimal synthetic building for $S_5$ is $G$-homotopy equivalent to $\Delta(\calS_2(G))$, which was calculated in \cite{Web1987-1} and has orbit diagram 
\begin{center}
\begin{tikzpicture}
    \fill (0,0) circle (0.1)
        node[left=1pt] {$S_4$};
    \fill (2,0) circle (0.1)
        node[right=1pt] {$C_2\times S_3$};
    \draw (0,0)--(2,0)
        node[below=1pt,pos=0.5] {$H$};
\end{tikzpicture}
\end{center}
where $H=\langle (1,2),(3,4)\rangle$.
\begin{proof}
    Let $\Delta$ be a non-trivial minimal synthetic building for $S_5$ at $p=2$, so that $S_5$ is not the stabilizer of any simplex of $\Delta$. The subgroups $S_4=\langle (1,2,3,4),(1,2)\rangle$ and $C_2\times S_3=\langle(1,2),(3,4),(3,4,5)\rangle$ have non-identity normal 2-subgroups, so they have contractible (and hence non-empty) fixed points on $\Delta$, as does their intersection $H=\langle (1,2),(3,4)\rangle$. Because $S_4$ and $C_2\times S_3$ are maximal in $S_5$ it follows that there are two vertices in $\Delta$, one with stabilizer equal to $S_4$, the other with stabilizer equal to $C_2\times S_3$. Because $\Delta^H\simeq\bullet$ there is a path of vertices and edges in $\Delta^H$ joining them. Choose a shortest such path. Excluding the end vertex stabilizers, the other stabilizers must be  $H$ or conjugates of $D_8$, because these account for all the subgroups containing $H$. If there is no $D_8$ stabilizer then the orbits of simplices in this path form a subcomplex equivariantly homotopy equivalent to $\Delta(\calS_2(G))$, and by minimality of $\Delta$ this subcomplex is the whole of $\Delta$. If there is a vertex $x$ with stabilizer of type $D_8$ on this path then because $\Delta^{D_8}\simeq \bullet$ there is a path with $D_8$ stabilizers from $x$ to the vertex with stabilizer $S_4$. From this we can construct a path with stabilizers that start at $S_4$, then are $D_8$ for some edges and vertices, then are $H$ for the remaining edges and vertices until the end vertex with stabilizer $C_2\times S_3$. Again, the orbits of simplices on this path form a subcomplex equivariantly homotopy equivalent to $\Delta(\calS_2(G))$, and by minimality of $\Delta$ this is the whole of $\Delta$.  
\end{proof}


\subsection{Families of infinitely many synthetic buildings}
\label{infinite-family-subsection}

We can obtain an infinite family of slender synthetic buildings of distinct homotopy types of bounded dimension for any finite group whose subgroups meet certain conditions. This is surprising in view of the difficulty, in general, of finding small exotic synthetic buildings. Once again, strongly $p$-embedded subgroups play a crucial role. 
In the next result we will use the notation
$$N_G(H,K)=\{x\in G\bigm| {}^xH\subseteq K\}$$
for the \textit{transporter} of the subgroup $H$ into the subgroup $K$.

\begin{theorem}\label{thm:infiniteGgraphs}
    Let $G$ be a finite group that contains a strongly $p$-embedded subgroup $S$. Furthermore, assume that there exists a subgroup $T$ such that for all elements $g\in S$ of order $p$ we have $N_G(\langle g\rangle,T)\subseteq TS$ with $p \bigm| |S\cap T|$ and $T\not\subseteq S$. Then there exists an infinite family of $1$-dimensional slender synthetic buildings for $G$ at the prime $p$, all of distinct homotopy types, with orbit diagrams as shown.

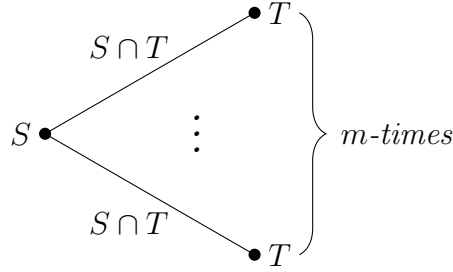
\begin{figure}[H]
\label{fig:pinwheelG}
    \begin{center}
    \begin{tikzpicture}[scale=0.8]
    \fill (3.46,2) circle (0.1)
        node[right=1pt] {$T$};
    \fill (0,0) circle (0.1)
        node[left=1pt] {$S$};
    \draw (0,0)--(3.46,2)
        node[above=6pt,pos=0.4] {$S\cap T$};
    \fill (3.46,-2) circle (0.1)
        node[right=1pt] {$T$};
    \draw (0,0)--(3.46,-2)
        node[below=7pt,pos=0.4] {$S\cap T$};
    \node[scale=1.5] at (2.5,0.2) {$\vdots$};
    \draw [decorate,decoration={brace,amplitude=10pt},rotate=270] (-2,4.2) -- (2,4.2);
    \node at (5.8,0) {$m$-times};
    \end{tikzpicture}
    \end{center}
    \end{figure}
\end{theorem}

\begin{proof}
There is a unique such graph $\Gamma$ with this orbit diagram by Theorem~\ref{thm:uniqueGgraphs}. We may identify the vertices of $\Gamma$ with the $G$-sets $G/S$ and copies of $G/T$, and the edges with copies of the $G$-set $G/(S\cap T)$. Because $p$ divides $|S\cap T|$, all stabilizers of vertices and edges in $\Gamma$ have order divisible by $p$. 

By construction, $G\bs\Gamma$ is contractible. To verify that $\Gamma$ is a synthetic building it remains to verify the condition from Proposition~\ref{equiv-condition-proposition}: $\Delta^H$ is ordinarily contractible for all subgroups $H\le G$ with $O_p(H)\ne 1$. For this, it suffices by Lemma~\ref{lem:contractible_pelems} to show that $\Gamma^{\langle g\rangle}$ is contractible for all elements $g\in G$ of order $p$. 
    
Let $g\in G$ be an element of order $p$.  Because $S$ is strongly $p$-embedded in $G$, there is a single vertex in the orbit of type $G/S$ fixed by $g$. We may assume, without loss of generality, that $g\in S$ so that this fixed vertex is the coset $S$. This is because every element of $G$ of order $p$ is conjugate to an element of $S$, by the strongly $p$-embedded condition, and conjugate elements have homeomorphic fixed points.

If $m=0$, meaning there are no orbits of vertices of the form $G/T$, then $\Gamma^{\langle g\rangle}$ is the single vertex corresponding to the coset $S$ and is hence contractible. Now suppose $m\geq1$. Consider an orbit of vertices of the form $G/T$ and its incident orbit of edges $G/(S\cap T)$. 
The fixed points of $g$ on $G/T$ are $\{xT\bigm| g\in {}^xT\}$ so such $x$ with $xT$ fixed have $x^{-1}\in N_G(\langle g\rangle,T) \subseteq TS$. Thus $x\in ST$ and $x=st$ for some $s\in S$ and $t\in T$, so the fixed coset $xT = stT = sT$. We deduce that $xT$ is an end vertex of the edge $s(S\cap T)$, whose other end vertex is $sS = S$. We deduce that each vertex in $\Gamma^{\langle g\rangle}$ is either the single coset $S$, or a coset $xT$ that is joined by a unique edge to $S$. If $r$ is the number of vertices in $G/T$ fixed by $g$, then $\Gamma^{\langle g\rangle}$ consists of a single vertex with stabilizer $S$ incident with $r\cdot m$ vertices of type $T$ via $r\cdot m$ edges of type $S\cap T$, as shown. 
    
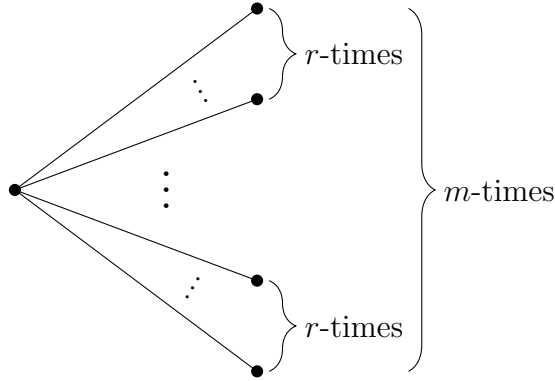
\begin{center}
    \begin{tikzpicture}[scale=0.8]
    \fill (4,3) circle (0.1);
    \fill (0,0) circle (0.1);
    \draw (0,0)--(4,3);
    \fill (4,-3) circle (0.1);
    \draw (0,0)--(4,-3);
    \fill (4,1.5) circle (0.1);
    \draw (0,0)--(4,1.5);
    \fill (4,-1.5) circle (0.1);
    \draw (0,0)--(4,-1.5);
    \node[rotate=30] at (3,1.75) {$\vdots$};
    \node[rotate=330] at (3,-1.5) {$\vdots$};
    \node[scale=1.5] at (2.5,0.2) {$\vdots$};
    \draw [decorate,decoration={brace,amplitude=10pt},rotate=270] (-3,6.5) -- (3,6.5);
    \node at (8,0) {$m$-times};
    \draw [decorate,decoration={brace,amplitude=10pt},rotate=270] (1.5,4.2) -- (3,4.2);
    \node at (5.6,2.25) {$r$-times};
    \draw [decorate,decoration={brace,amplitude=10pt},rotate=270] (-3,4.2) -- (-1.5,4.2);
    \node at (5.6,-2.25) {$r$-times};
    \end{tikzpicture}
    \end{center}
    \noindent This fixed point space is contractible, completing the argumnent that $\Gamma$ is a slender synthetic building for $G$ at $p$.

    Finally we show that these graphs have different homotopy types for different choices of $m$. The Euler characteristic of $\Gamma$ is
    $$\chi(\Gamma)=|G/S|+m\cdot|G/T|-m\cdot|G/(S\cap T)|$$
    Note that $|G/T|<|G/(S\cap T)|$ since $T\ne (S\cap T)$, so the value of $\chi(\Gamma)$ decreases as $m$ increases. By ranging over all $m\in\NN$, we obtain infinitely many distinct homotopy types for $\Gamma$. 
    \end{proof}

We have just shown how to construct an infinite family of slender synthetic buildings of distinct homotopy types for a fixed group $G$ that is not obtained via taking suspensions, because they all have dimension 1.
It sometimes happens that there exist multiple non-conjugate subgroups $T\leq G$ with the properties in the statement of Theorem~\ref{thm:infiniteGgraphs}, and then we can achieve more distinct homotopy types of this kind. This is exemplified for $A_5$ at $p=2$.

\begin{example}
\label{a5-families-example}
    Let $G=A_5$ and $p=2$, and take $S=A_4$, which is strongly 2-embedded in $G$. We can choose that either $T=S_3$ or $T=D_{10}$ in such a way that $S\cap T=C_2$. We observe that if $g\in S$ is an element of order 2 then $N_G(\langle g \rangle, T) = TN_G(\langle g\rangle)\subseteq TS$ because all elements of order 2 in $T$ are conjugate in $T$ and normalizers (= centralizers) in $G$ of subgroups of $S$ of order 2 are contained in $S$.
    Thus Theorem~\ref{thm:infiniteGgraphs} applies and we obtain infinite families of exotic synthetic buildings for $A_5$ using either choice of $T$. 

    Because there are two choices for the subgroup $T$ up to conjugacy class, we can construct infinitely many $A_5$-graphs of distinct homotopy types in the following way. Take a single orbit of vertices with stabilizers of type $A_4$. To this orbit of vertices, adjoin $m\in\NN$ orbits of vertices with stabilizers of type $S_3$ each via a single orbit of edges with stabilizers of type $C_2$, where $C_2=A_4\cap S_3$. By Theorem~\ref{thm:uniqueGgraphs}, this construction is unique up to equivariant homotopy equivalence. In the same manner, adjoin $n\in\NN$ orbits of vertices with stabilizers of type $D_{10}$ each via a single orbit of edges with stabilizers of type $C_2$, where $C_2=A_4\cap D_{10}$. The resulting orbit diagram for $\Gamma$ is shown. 
    
    
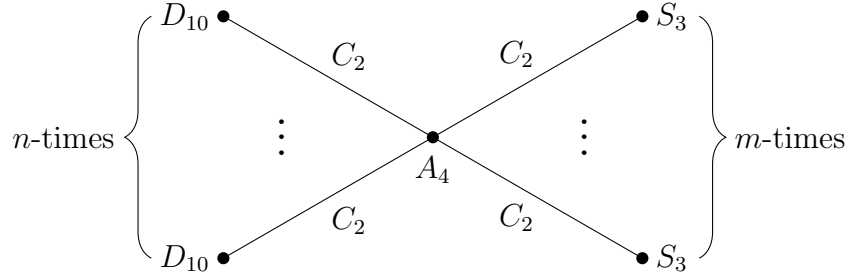
\begin{center}
    \begin{tikzpicture}[scale=0.8]
    \fill (3.46,2) circle (0.1)
        node[right=1pt] {$S_3$};
    \fill (0,0) circle (0.1)
        node[below=3pt] {$A_4$};
    \fill (-3.46,2) circle (0.1)
        node[left=1pt] {$D_{10}$};
    \draw (0,0)--(3.46,2)
        node[above=4pt,pos=0.4] {$C_2$};
    \draw (0,0)--(-3.46,2)
        node[above=3pt,pos=0.4] {$C_2$};
    \fill (3.46,-2) circle (0.1)
        node[right=1pt] {$S_3$};
    \fill (-3.46,-2) circle (0.1)
        node[left=1pt] {$D_{10}$};
    \draw (0,0)--(3.46,-2)
        node[below=3pt,pos=0.4] {$C_2$};
    \draw (0,0)--(-3.46,-2)
        node[below=4pt,pos=0.4] {$C_2$};
    \node[scale=1.5] at (2.5,0.2) {$\vdots$};
    \node[scale=1.5] at (-2.5,0.2) {$\vdots$};
    \draw [decorate,decoration={brace,amplitude=10pt},rotate=90] (-2,4.65) -- (2,4.65);
    \node at (-6.1,0) {$n$-times};
    \draw [decorate,decoration={brace,amplitude=10pt},rotate=270] (-2,4.4) -- (2,4.4);
    \node at (5.9,0) {$m$-times};
    \end{tikzpicture}
    \end{center}
    Euler characteristic of $\Gamma$ is
    $$\chi(\Gamma)=5+m\cdot(10-30)+n\cdot(6-30)=5-20m-24n.$$
    Ranging $m$ and $n$ over $\ZZ_{\geq0}$ yields infinitely many possible values for $\chi(\Gamma)$, and hence infinitely many distinct homotopy types for $\Gamma$. When $m=1$ and $n=0$ the graph $\Gamma$ is shown in Example~\ref{Delta13-example}.
    
    We note also that the reduced Lefschetz modules for these synthetic buildings change, along with the Euler characteristics. For example, the two synthetic buildings at $p=2$ with orbit diagrams
    \begin{center}
\begin{tikzpicture}
    \fill (-3.5,0) circle (0.1)
        node[left=1pt] {$A_4$};
\node at (-2,0) {and};
    \fill (0,0) circle (0.1)
        node[left=1pt] {$A_4$};
    \fill (2,0) circle (0.1)
        node[right=1pt] {$S_3$};
    \draw (0,0)--(2,0)
        node[below=1pt,pos=0.5] {$C_2$};
\end{tikzpicture}
\end{center}
have reduced Lefschetz modules $P_4$ and $-(P_{2a}+P_{2b})$ over a splitting field of characteristic 2, where $P_4$ is the 4-dimensional irreducible projective module and $P_{2a},P_{2b}$ are the projective covers of the two 2-dimensional irreducible modules.

\end{example}
\begin{example}
We can find many more examples of groups $G$ with subgroups $S$ and $T$ satisfying the conditions of Theorem~\ref{thm:infiniteGgraphs} in the following way. Take primes $p,q$ with $q=2p+1$ (such as $(2,5)$ or $(3,7)$ or $(5,11)$) and let $G=C_q\rtimes C_{2p}$. Take $S=C_{2p}$ (which is strongly $p$-embedded in $G$) and $T=C_q\rtimes C_p$. We immediately see that the conditions of Theorem~\ref{thm:infiniteGgraphs} are satisfied, noting that $TS=G$ in this situation. The graphs $\Gamma$ on which $G$ acts are complete bipartite graphs connecting $q$ vertices to $2m$ vertices, where there are $m$ orbits of vertices of type $T$.
\end{example}

A feature of this construction is that the entire group $G$ does not appear as a simplex stabilizer, leading to non-trivial equations in group cohomology of the kind shown in Theorem~\ref{thm:WebbA}. If we consider slender synthetic buildings where $G$ does appear, we can achieve another infinite family with fewer conditions on the subgroups of $G$ in the following way.
\begin{theorem}
Let $G$ be a finite group that contains a strongly $p$-embedded subgroup. Then there is an infinite family of 1-dimensional slender synthetic buildings for $G$ at the prime $p$, all of distinct homotopy types.
\end{theorem}

\begin{proof}
Define $\Gamma_n$ to be a $G$-graph whose orbit diagram is a tree with $n$ orbits of vertices with stabilizer $G$ itself and $n-1$ orbits of edges whose stabilizers are all the strongly $p$-embedded subgroup $S$. In particular, if $n=1$ then the orbit diagram is a single orbit of vertices. By construction, $G\bs\Gamma_n$ is contractible. Furthermore, every stabilizer subgroup has order divisible by $p$. Since $S$ is strongly $p$-embedded, every orbit of edges contains a single fixed edge under the action of a subgroup $H\leq G$ with $O_p(H)\ne1$. Clearly every vertex is fixed by the action of every subgroup of $G$, so the fixed point space $\Gamma_n^H$ is isomorphic to $G\bs\Gamma_n$ and hence contractible. Thus, $\Gamma_n$ is a slender synthetic building for $G$ at $p$ for all $n\geq1$. The homotopy type of $\Gamma_n$ is distinct for all $n$ since an increase in $n$ by 1 results in $|G|/|S|-1$ more circuits.
\end{proof}


\section{Constructions for synthetic buildings}
\label{constructions-section}
Ideally, we would wish to construct as many synthetic buildings as we can and, if possible, classify them in their entirety. With this in mind we consider operations on $G$-simplicial complexes that can produce new synthetic buildings from those that are already known.

\subsection{The functor from weak synthetic buildings to synthetic buildings}

\begin{definition}
    Given a weak synthetic building $\Delta$ for $G$ at the prime $p$, we define $\mathrm{SB}(\Delta)$ to be the subset of the simplices of $\Delta$ consisting of the simplices whose stabilizers have order divisible by $p$.
\end{definition}

\begin{proposition}\label{prop:weak_to_strong}
    Let $\Delta$ be a weak synthetic building for $G$ at the prime $p$. Then $\mathrm{SB}(\Delta)$ is a synthetic building for $G$ at $p$. It is the unique maximal subcomplex of $\Delta$ with this property.
\end{proposition}
\begin{proof}
    Since $\Delta$ is a weak synthetic building, we have that $\Delta^P$ is $N_G(P)$-contractible for all non-identity $p$-subgroups $P$ of $G$, and it equals $\mathrm{SB}(\Delta)^P$, so that $\mathrm{SB}(\Delta)^P$ is $N_G(P)$-contractible as well. All simplices in $\mathrm{SB}(\Delta)$ have stabilizers of order divisible by $p$ by construction. Evidently $\mathrm{SB}(\Delta)$ is the unique maximal subcomplex that is a synthetic building for $G$ at $p$.
\end{proof}

\begin{corollary}\label{cor:weakSlend_to_Slend}
    If $\Delta$ is a graph that is a slender weak synthetic building for $G$ at the prime $p$, then $\mathrm{SB}(\Delta)$ is slender also.
\end{corollary}
\begin{proof}
    By Proposition~\ref{prop:sb_ctd_orbit_space}, the orbit space $G\bs \mathrm{SB}(\Delta)$ is connected. Furthermore, since $G\bs \mathrm{SB}(\Delta)$ is a connected subgraph of the tree $G\bs\Delta$, necessarily $G\bs \mathrm{SB}(\Delta)$ is also a tree. Hence $\mathrm{SB}(\Delta)$ is slender. 
\end{proof}

\begin{proposition}
\label{SB-is-functorial-proposition}
    The operation $\mathrm{SB}$ defines a functor from the category of weak synthetic buildings with equivariant simplicial maps to the category of synthetic buildings with equivariant simplicial maps.
\end{proposition}

\begin{proof}
    What remains to be shown is that if $f:\Delta_1\to\Delta_2$ is a $G$-simplicial map between weak synthetic buildings then $f(\mathrm{SB}(\Delta_1))\subseteq \mathrm{SB}(\Delta_2)$. This is so because if $\sigma\in\Delta_1$ is a simplex then its stabilizer satisfies $G_\sigma\subseteq G_{f(\sigma)}$ so that if $p\bigm| |G_\sigma|$ then $p\bigm| |G_{f(\sigma)}|$. 
\end{proof}

\subsection{Restriction of synthetic buildings}
\begin{definition}
    If $\Delta$ is a weak synthetic building for $G$ and $H\le G$ we write $\Delta\downarrow_H^G$ for $\Delta$ regarded as an $H$-simplicial complex, and this is a weak synthetic building for $H$. Even if we start with a synthetic building $\Delta$ for $G$ it might not be the case that $\Delta\downarrow_H^G$ is a synthetic building for $H$. We define the  \textit{restriction} of the synthetic building $\Delta$ from $G$ to $H$ to be $\Res_H^G(\Delta)=\SB(\Delta\downarrow_H^G)$, and this is again a synthetic building.
\end{definition}

\begin{proposition}
    If $K\le H\le G$ then $\Res_K^G(\Delta) = \Res_K^H\Res_H^G(\Delta)$.
\end{proposition}

\begin{proof}
    This is an immediate check.
\end{proof}

\subsection{The join}
\begin{definition}
Let $G_1$ and $G_2$ be groups, let $\Delta_1$ be an abstract $G_1$-simplicial complex and $\Delta_2$ be an abstract $G_2$-simplicial complex, with $\Delta_1,\Delta_2$ non-empty. Their \textit{join}, $\Delta_1*\Delta_2$, is a $G_1\times G_2$-simplicial complex whose vertices are the disjoint union of the vertices of $\Delta_1$ and $\Delta_2$, and whose simplices are the subsets of the form $a$, $b$ and $a\cup b$  where $a\in\Delta_1$ and $b\in\Delta_2$. Evidently, $\Delta_1*\Delta_2$ acquires an action of $G_1\times G_2$ where $(g_1,g_2)a=g_1a$, $(g_1,g_2)b=g_2b$ and $(g_1,g_2)(a\cup b)=g_1a\cup g_2b$. The set-wise stabilizer of a simplex such as $a\cup b$ fixes its vertices pointwise, because the same is true of the actions of $G_1,G_2$ on $\Delta_1,\Delta_2$, and so we obtain a $G_1\times G_2$-simplicial complex.
\end{definition}
In case $G_1=G_2=G$ we get an action of $G$ on $\Delta_1*\Delta_2$ via the diagonal embedding of $G$ in $G\times G$.
\begin{theorem}
    Let $\Delta_1$ be a $G_1$-simplicial complex and $\Delta_2$ a $G_2$-simplicial complex, for groups $G_1$ and $G_2$ respectively. Suppose that $\Delta_1,\Delta_2$ are non-empty.
    \begin{enumerate}
        \item If $\Delta_1$ and $\Delta_2$ are weak synthetic buildings, then $\Delta_1*\Delta_2$ is a weak synthetic building for $G_1\times G_2$.
        \item If $\Delta_1$ and $\Delta_2$ are synthetic buildings, then $\Delta_1*\Delta_2$ is a synthetic building for $G_1\times G_2$. 
    
    \end{enumerate}
    
\end{theorem}

\begin{proof}
    (1) To verify the fixed point condition for a weak synthetic building, let $p_i:G_1\times G_2\to G_i$ where $i=1,2$ be the two projections onto the factor groups. The fixed points of a subgroup $K$ of $G_1\times G_2$ on $\Delta_1*\Delta_2$ are $\Delta_1^{p_1(K)}*\Delta_2^{p_2(K)}$ and if $O_p(K)\ne 1$ then at least one of $O_p(p_1(K))\ne 1$ or $O_p(p_2(K))\ne 1$, so that at least one factor in the fixed point join is contractible, and that forces the fixed points to be contractible.
    
    (2) Suppose $\Delta_1$ and $\Delta_2$ are synthetic buildings for $G_1$ and $G_2$. The stabilizer of $a\cup b$ in $G_1\times G_2$ is $\Stab_{G_1}(a)\times \Stab_{G_2}(b)$, and this has order divisible by $p$, because the stabilizers of $a$ and $b$ do. Similarly the stabilizer of $a$ in $G_1\times G_2$ is $\Stab_{G_1}(a)\times G_2$ and the stabilizer of $b$ in $G_1\times G_2$ is $G_1\times \Stab_{G_2}(b)$, and both these have order divisible by $p$. This verifies the second condition in the definition of a synthetic building, and the first condition was verified in (1). 

\end{proof}
\begin{example}
    As an application of the join construction, we may construct groups with minimal synthetic buildings in a range of different dimensions. In Example~\ref{two-dimensions-example} we observed that each of the groups $S_{2p}$ when $p\ge 5$ has a non-trivial minimal synthetic building $\Delta_0$ of dimension 0, and another minimal synthetic building $\Delta_1$ of dimension 1. Taking the three joins $\Delta_0*\Delta_0$, $\Delta_0*\Delta_1$ and $\Delta_1*\Delta_1$ we obtain minimal synthetic buildings for $G=S_{2p}\times S_{2p}$ of dimensions 1, 2 and 3. We may verify, for example in the case $p=5$, that they are all minimal by techniques such as listing the $G$-subcomplexes with connected orbit diagram (exploiting Proposition~\ref{prop:sb_ctd_orbit_space}) and showing that none of them has Euler characteristic $\chi \equiv 1 \;(\textrm{mod}\; p^4)$, as well as the observation that if a join $\Delta_1*\Delta_2$ is a synthetic building for $G_1\times G_2$, where each $\Delta_i$ is a complex for $G_i$, then each $\Delta_i$ must be a synthetic building for $G_i$.
\end{example}

We can use the join to get an operation on (weak) synthetic buildings for a single group $G$, making the set of $G$-homotopy classes of such buildings into a semigroup. If $\Delta_1,\Delta_2$ are weak synthetic buildings for $G$ then $\Delta_1*\Delta_2$ becomes a weak synthetic building on letting $G$ act via the diagonal homomorphism $\delta:G\to G\times G$ given by $\delta(g)=(g,g)$. Thus $(\Delta_1*\Delta_2)\downarrow_{\delta(G)}^{G\times G}$ is a weak synthetic building for $G$. We can convert it into a synthetic building $\Res_{\delta(G)}^{G\times G}(\Delta_1*\Delta_2)=\SB((\Delta_1*\Delta_2)\downarrow_{\delta(G)}^{G\times G})$ for $G$. 

\begin{proposition}
    Given synthetic buildings $\Delta_1,\Delta_2$ for $G$, the operation $(\Delta_1,\Delta_2)\mapsto \Res_{\delta(G)}^{G\times G}(\Delta_1*\Delta_2)$ is an associative binary operation on $G$-equivariant homotopy equivalence classes of synthetic buildings for $G$. Thus the set of $G$-equivariant homotopy equivalence classes of synthetic buildings becomes a semigroup under this operation.
\end{proposition}

\begin{proof}
    It is straightforward to see that the join itself is associative: $(\Delta_1*\Delta_2)*\Delta_3\cong \Delta_1*(\Delta_2*\Delta_3)$ as $G$-simplicial complexes with the diagonal action. If we also remove orbits with stabilizers of order prime to $p$ in two stages, corresponding to the two joins, we get the same result as if we removed such orbits in one go from the two-step join. The fact that the operation respects $G$-homotopy classes arises because the join operation does, and the operation $\SB$ is functorial, by Proposition~\ref{SB-is-functorial-proposition}.
\end{proof}

\subsection{The product}
We can also produce a binary operation on synthetic buildings using the product. 
Given a $G_1$-simplicial complex $\Delta_1$ and a $G_2$-simplicial complex $\Delta_2$, we may give a simplicial action of $G_1\times G_2$ on the product space $\Delta_1\times \Delta_2$ as follows.
For each $i$, let $\calP(\Delta_i)$ be the poset of simplices of $\Delta_i$, ordered by inclusion. Thus $\calP(\Delta_i)$ is acted on by $G_i$ and its order complex is the barycentric subdivision of $\Delta_i$. We have seen in Proposition~\ref{subdivision-proposition} that passing to a subdivision does not affect being a synthetic building or minimality. Now $G_1\times G_2$ acts on the product poset $\calP(\Delta_1)\times \calP(\Delta_2)$ and hence on its order complex, which is a triangulation of the space $\Delta_1\times\Delta_2$. We let $\Delta_1\times\Delta_2$ also denote the simplicial complex we have just defined.

It will often be the case for synthetic buildings $\Delta_1,\Delta_2$ that $\Delta_1\times \Delta_2$ is not even a weak synthetic building for $G_1\times G_2$, and small examples such as $G_1=G_2=S_3$ at $p=2$ and $\Delta_1=\Delta_2 = \{\hbox{three points}\}$ show this. On the other hand, if $G_1=G_2=G$ and we let $G$ act diagonally we have the following.

\begin{proposition}
    Let $\Delta_1,\Delta_2$ be weak synthetic buildings for $G$. Then $\Delta_1\times \Delta_2$ is a weak synthetic building for $G$ with the diagonal action. Thus $\SB(\Delta_1\times \Delta_2)$ is a synthetic building for $G$, and this provides an associative binary operation on $G$-equivariant homotopy classes of synthetic buildings for $G$. If $G$ has order divisible by $p$, the trivial synthetic building acts as the identity for this operation.
\end{proposition}

\begin{proof}
We check that $\Delta_1\times \Delta_2$ is a weak synthetic building when $G$ has the diagonal action. If $H\le G$ then $(\Delta_1\times \Delta_2)^H=\Delta_1^H\times \Delta_2^H$ and if $O_p(H)\ne 1$ so that $\Delta_1^H$ and $\Delta_2^H$ are contractible, then so is $(\Delta_1\times \Delta_2)^H$. Thus $\Delta_1\times \Delta_2$ is a weak synthetic building. The fact that the operation $\SB(\Delta_1\times \Delta_2)$ is associative comes from the fact that the product itself is associate, and if we remove orbits whose stabilizers have order prime to $p$ in two steps, we get the same as if we had removed them from $\Delta_1\times \Delta_2\times\Delta_3$ in one step.

To see that this binary operation respects $G$-equivariant homotopy classes, if $\Delta_1\simeq_G\Delta'_1$ are $G$-homotopy equivalent weak synthetic buildings, the maps that show this extend to give $\Delta_1\times\Delta_2\simeq_G\Delta'_1\times\Delta_2$. Applying the functor $\SB$ to these maps gives a $G$-homotopy equivalence in the category of synthetic buildings.  It is immediate that the trivial synthetic building acts as the identity.
\end{proof}

\begin{corollary}
    Under the operation of product followed by the operation $\SB$, the set of $G$-equivariant homotopy classes of synthetic buildings becomes a monoid.
\end{corollary}

\begin{example}
Several properties of the join and product are illustrated by considering the situation when $p=2$ and  the pair of groups $G_1, G_2$ are both isomorphic to $S_3$, with $\Delta_1=\Delta_2 = S_3/C_2$, a set of three points. In this situation it is readily calculated that the join $\Delta_1*\Delta_2$ is a minimal synthetic building for $S_3\times S_3$. However, $(\Delta_1*\Delta_2)\downarrow_{\delta(S_3)}^{S_3\times S_3}$ is a weak synthetic building for $S_3$, but it is not a synthetic building because it has a free orbit of edges.  It is also not slender. Its orbit diagram is first of the two displayed below, and on applying the functor $\SB$ we obtain the synthetic building $\Res_{\delta(S_3)}^{S_3\times S_3}(\Delta_1*\Delta_2)$ whose orbit diagram is the second shown. This is homotopy equivalent to the original sets $\Delta_i$, so that $\Delta_i$ is idempotent under the binary operation given by join.
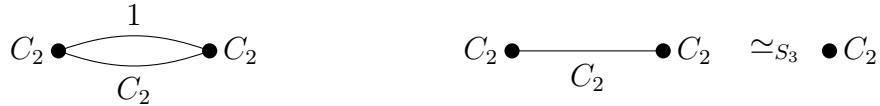
\begin{center}
\begin{tikzpicture}
    \fill (-2,0) circle (0.1)
        node[left=1pt] {$C_2$};
    \fill (0,0) circle (0.1)
        node[right=1pt] {$C_2$};
    \draw (-2,0) to [out=20,in=160] (0,0);
    \node[above] at (-1,.2) {$1$};
    \draw (-2,0) to [out=340,in=200] (0,0);
    \node[below] at (-1,-0.2) {$C_2$};
    \fill (4,0) circle (0.1)
        node[left=1pt] {$C_2$};
    \fill (6,0) circle (0.1)
        node[right=1pt] {$C_2$};
    \draw (4,0)--(6,0)
        node[below=1pt,pos=0.5] {$C_2$};
    \node[right] at (7,0) {$\simeq_{S_3}$};
 \fill (8.2,0) circle (0.1)
    node[right=1pt] {$C_2$};
\end{tikzpicture}
\end{center}

If we consider the product $\Delta_1\times\Delta_2$, it is a set of 9 points and is not even a weak synthetic building for $S_3\times S_3$. For $\delta(S_3)$ it is a weak synthetic building that is the disjoint union of a free orbit and an orbit of size 3. Applying the functor $\SB$ we see that $\Res_{\delta(S_3)}^{S_3\times S_3}(\Delta_1\times\Delta_2)$ is the single orbit of size 3, so that $\Delta_i$ is idempotent with respect to the binary operation given by product (as well as with respect to the binary operation given by join).

We see from this that neither the semigroup given by join nor the monoid given by product embed in a group.
\end{example}

\subsection{The suspension}
We take the \textit{suspension} of a $G$-simplicial complex $\Delta$ to be the join
$$\Sigma(\Delta)=\Delta*\{\{a\},\{b\}\}$$
where $a\not=b$ do not already appear as vertices of $\Delta$, and $G$ acts trivially on them.

\begin{proposition}
    The suspension $\Sigma$ preserves synthetic buildings, weak synthetic buildings, and the property of being slender.
\end{proposition}

\begin{proof}
    The fixed points of a subgroup on the suspension is the suspension of the fixed points, and if the fixed points are contractible so is their suspension. The stabilizers that arise with $\Sigma(\Delta)$ are the same as those that arise with $\Delta$, together with $G$ itself if $p\bigm| |G|$. Finally, the quotient of the suspension by $G$ is the suspension of the quotient, so the property of being slender is preserved.
\end{proof}

 The suspension provides a way to construct infinitely many synthetic buildings from non-trivial, non-empty synthetic buildings. Provided $G$ has order divisible by $p$, the proper suspensions will never be minimal because they always have the trivial synthetic building as a sub-synthetic building, so this construction is not useful in constructing exotic minimal synthetic buildings, and in fact their dimensions will increase without bound.

\section{A method for computing reduced Lefschetz modules}
\label{Lefschetz-module-section}
    
At several places where we describe specific synthetic buildings we have included the values of the reduced Lefschetz modules (defined in Theorem~\ref{structure-theorem-for-chain-complex}), but without details of the calculations. We present here a lemma that facilitates these calculations. A special case of this lemma appeared in \cite[Propn. 5.3]{Web1987-2} but no proof was given there. We provide a proof now.

The result is stated for a field $k$ of characteristic $p$, but it also holds if we assume $k$ is a complete discrete valuation ring with residue field of characteristic $p$, by the correspondence of projective $kG$-modules over the discrete valuation ring, and over the residue field.

\begin{lemma} Let $k$ be a  field of characteristic $p$ and $\Delta$ a weak synthetic building. For each simple
$kG$-module $S$, the multiplicity of the projective cover $P_S$
as a summand of $\tilde L(\Delta)$ equals
$$
\frac{\sum_{\sigma\in[\tilde\Delta/G]}(-1)^{\dim\sigma}
(\hbox{multiplicity of }P_k\hbox{ for }G_\sigma\hbox{ in
}S\downarrow_{G_\sigma})}
{\dim\End_{kG}(S)}.$$
\end{lemma}

Here $\tilde\Delta$ denotes $\Delta$ together with an artificial extra simplex stabilized by $G$ and in dimension $-1$. By the `multiplicity of $P_k$ for $G_\sigma$ in $S\downarrow_{G_\sigma}$' we mean the multiplicity as a direct summand. In case $\sigma$ is the extra simplex in degree $-1$, so that $G_\sigma = G$, the multiplicity of $P_k$ for $G$ in $S$ is only non-zero when $S=P_k$, which happens if and only if $|G|$ is prime to $p$.

\begin{proof}
We use the bilinear forms
$$\begin{aligned} (V,W)_G&=\dim_k\Hom_{kG}(V,W)\cr
\langle V,W\rangle_G&=\hbox{ multiplicity of $P_k$ for $kG$ as
a summand of }V^*\otimes W\cr
\end{aligned}$$
defined on the Green ring of $kG$-modules and considered on page 34 of \cite{Ben1984}.
When $V$ is a virtual projective module we have $(V,W)_G=\langle
V,W\rangle_G$ as can be deduced from 2.4.1 of
\cite{Ben1984}.

By Theorem~\ref{structure-theorem-for-chain-complex}, $\tilde L(\Delta)$ is a virtual projective module so we can write
$\tilde L(\Delta)=\sum_{S\,{\rm simple}}n_S\cdot P_S$ in the Green ring. We have
$$(\tilde L(\Delta),S)_G=n_S\cdot\dim\End_{kG}(S)$$
and this equals
$$
\begin{aligned}\langle \tilde L(\Delta), S\rangle_G
&=\langle \sum_{\sigma\in[\tilde\Delta/G]}(-1)^{\dim\sigma}
k\uparrow_{G_\sigma}^G, S\rangle_G\cr
&=\sum_{\sigma\in[\tilde\Delta/G]}(-1)^{\dim\sigma}\langle 
k\uparrow_{G_\sigma}^G, S\rangle_G\cr
&=\sum_{\sigma\in[\tilde\Delta/G]}(-1)^{\dim\sigma}\langle 
k, S\downarrow_{G_\sigma}^G\rangle_{G_\sigma}
\qquad\hbox{by \cite[Cor. 2.4.6]{Ben1984}}\cr
&=\sum_{\sigma\in[\tilde\Delta/G]}(-1)^{\dim\sigma}\hbox{
multiplicity of $P_k$ for $G_\sigma$ in }
S\downarrow_{G_\sigma}^G.\cr
\end{aligned}
$$
Dividing by $\dim\End_{kG}(S)$ gives us the multiplicity $n_S$.
\end{proof}

In the following corollary we write $|G|_p$ for the order of a Sylow $p$-subgroup of $G$.

\begin{corollary}
    The multiplicity of $P_S$ in $\tilde L(\Delta)$ is 0 unless $\dim S \ge \min_{\sigma\in\Delta}|G_\sigma|_p$. Thus if $\Delta$ is a synthetic building and $p\bigm| |G|$ then $P_k$ does not appear as a summand of  $\tilde L(\Delta)$.
\end{corollary}

\begin{proof}
    The dimension of $P_k$ for $kG_\sigma$ is at least $|G_\sigma|_p$ so that if $P_k$ is a summand of $S\downarrow_{G_\sigma}^G$ then $\dim S \ge |G_\sigma|_p$. When $\Delta$ is a synthetic building and $p\bigm| |G|$ every simplex $\sigma$ has 
    $p\bigm| |G_\sigma|$, so $\dim S \ge p$ if $P_S$ is a summand of $\tilde L(\Delta)$. This means $P_k$ cannot occur.
\end{proof}


\section{Conjectures}
\label{conjectures-section}
Many questions arise in this work, and we offer the following conjectures.

\begin{conjecture}
\label{finiteness-conjecture}
    For each finite group and prime $p$, there are only finitely many equivariant homotopy types of minimal synthetic buildings.
\end{conjecture}

\begin{conjecture}
\label{p-subgroups-complex-minimal-conjecture}
    For each finite group $G$ and prime $p$, the $p$-subgroups complex $\SpCpx{G}{p}$ has the equivariant homotopy type of a minimal synthetic building. Furthermore, the dimension of this minimal synthetic building is an upper bound for the dimensions of all minimal synthetic buildings for $G$ at $p$.
\end{conjecture}

The truth of Conjecture~\ref{p-subgroups-complex-minimal-conjecture} would be interesting, in that the dimension of a minimal synthetic building equivariantly homotopy equivalent to $\SpCpx{G}{p}$, plus 1, would provide an analogue, valid for all finite groups and primes, of the Lie rank of a group of Lie type.

\begin{conjecture}
\label{minimal-sb-are-images}
All minimal synthetic buildings are equivariantly homotopy equivalent to epimorphic images of $\SpCpx{G}{p}$.
\end{conjecture}

An affirmative answer to Conjecture~\ref{minimal-sb-are-images} would evidently imply the same for Conjecture~\ref{finiteness-conjecture}.

The next conjecture is about groups of Lie type in characteristic $p$, and it is true for groups of Lie rank at most 2, as follows from Corollary~\ref{Lie-rank-2-uniqueness} and Theorem~\ref{SPES-uniqueness-theorem}. Its truth in general would mean that for these groups we have characterized the Tits building by our topological definition of a synthetic building together with minimality.

\begin{conjecture}
\label{BN-pair-conjecture}
    For a group of Lie type in characteristic $p$, there is only one non-trivial minimal synthetic building up to equivariant homotopy equivalence, namely the Tits building.
\end{conjecture}

The case $\calY=\{1\}$ of the following conjecture says that if $\Delta$ is a synthetic building then $G\backslash\Delta$ is contractible or, in other words, synthetic buildings are slender. In case $\Delta = \Delta(\calS_p(G))$ this was shown to be true by Symonds~\cite{Sym1998}, with subsequent proofs given, for example, in \cite{Bux1999, Gro2023, Ste2023}. The evidence for Conjecture~\ref{orbit-space-contractible-conjecture} is Corollary~\ref{orbit-space-mod-p-acyclic-corollary}.

\begin{conjecture}
\label{orbit-space-contractible-conjecture}
Let $\Delta$ be a synthetic building for $G$ at the prime $p$, and let $\calY$ be a non-empty set of $p$-subgroups of $G$, closed under taking subgroups and conjugation. Suppose $\calY$ does not contain a Sylow $p$-subgroup of $G$ and put 
$$
\Delta_\calY = \{\sigma \in \Delta\bigm| \hbox{Sylow $p$-subgroups of } G_\sigma \hbox{ do not lie in } \calY\}.
$$
Then $\Delta_\calY$ is slender (meaning that $G\backslash \Delta_\calY$ is contractible).
\end{conjecture}

We conclude with three conjectures, the validity of any of which would imply the validity of Quillen's conjecture.

\begin{conjecture}
    Let $\Delta$ be a non-trivial synthetic building for $G$ at the prime $p$. Suppose one of the following conditions is satisfied:
\begin{enumerate}
    \item The reduced Lefschetz module of $\Delta$ over $\FF_p$ is zero: $\tilde L(\Delta)=0$.
    \item The augmented chain complex over $\FF_p$ is contractible: $\tilde C.(\Delta)\simeq_G 0$.
    \item $\Delta$ is ordinarily contractible: $\Delta\simeq \bullet$.
\end{enumerate}
Then $O_p(G)\ne 1$.
\end{conjecture}

Evidently condition (3) implies condition (2) (bearing in mind Theorem~\ref{structure-theorem-for-chain-complex}), which in turn implies condition (1).


\end{document}